\documentclass{amsart}
\usepackage{mathptmx}
\usepackage[T1]{fontenc}
\usepackage[latin9]{inputenc}
\usepackage[a4paper]{geometry}
\usepackage{refstyle}
\usepackage{units}
\usepackage{amsthm}
\usepackage{amsmath}
\usepackage{amssymb}
\usepackage[unicode=true,pdfusetitle,
 bookmarks=true,bookmarksnumbered=false,bookmarksopen=false,
 breaklinks=false,pdfborder={0 0 0},backref=false,colorlinks=true]
 {hyperref}
\hypersetup{
 pdfstartview=FitB,pdfstartpage=1,bookmarksopen}

\makeatletter

\AtBeginDocument{\providecommand\secref[1]{\ref{sec:#1}}}
\AtBeginDocument{\providecommand\thmref[1]{\ref{thm:#1}}}
\AtBeginDocument{\providecommand\remref[1]{\ref{rem:#1}}}
\AtBeginDocument{\providecommand\hyporef[1]{\ref{hypo:#1}}}
\AtBeginDocument{\providecommand\propref[1]{\ref{prop:#1}}}
\AtBeginDocument{\providecommand\exaref[1]{\ref{exa:#1}}}
\AtBeginDocument{\providecommand\eqref[1]{\ref{eq:#1}}}
\AtBeginDocument{\providecommand\corref[1]{\ref{cor:#1}}}
\AtBeginDocument{\providecommand\lemref[1]{\ref{lem:#1}}}
\RS@ifundefined{subref}
  {\def\RSsubtxt{section~}\newref{sub}{name = \RSsubtxt}}
  {}
\RS@ifundefined{thmref}
  {\def\RSthmtxt{theorem~}\newref{thm}{name = \RSthmtxt}}
  {}
\RS@ifundefined{lemref}
  {\def\RSlemtxt{lemma~}\newref{lem}{name = \RSlemtxt}}
  {}
\def\RSsectxt{Section~}%
\usepackage{latexsym, amsfonts, amssymb, amsmath}
\theoremstyle{plain}
\numberwithin{equation}{section}
\numberwithin{figure}{section}
\numberwithin{table}{section}
\usepackage{refstyle}
\usepackage{refstyle}
    \newtheorem{stthm}{\protect\theoremname}[section]
    
    \newtheorem{spprop}{\protect\propositionname}[section]

 \def\SEC[#1,#2,#3]{\renewcommand\section{\@startsection {section}{1}{\z@}%
 {#2ex \@plus -1ex \@minus -.2ex}%
 {#3ex \@plus.2ex}%
 {\reset@font #1}}}
 \newref{sec}{%
 name = \RSsectxt,
 names = \RSsecstxt,
 lsttxt = \RSlsttxt,
 lsttwotxt = \RSlsttwotxt,
 refcmd = \ref{#1}}
 \def\SEC[#1,#2,#3]{\renewcommand\section{\@startsection {section}{1}{\z@}%
 {#2ex \@plus -1ex \@minus -.2ex}%
 {#3ex \@plus.2ex}%
 {\reset@font #1}}}
 \newref{sec}{%
 name = \RSsectxt,
 names = \RSsecstxt,
 lsttxt = \RSlsttxt,
 lsttwotxt = \RSlsttwotxt,
 refcmd = \ref{#1}}
 \newref{thm}{%
 name = \theoremname~,
 lsttxt = \RSlsttxt,
 lsttwotxt = \RSlsttwotxt,
 refcmd = {\ref{#1}}}
 \usepackage{enumitem}
 \theoremstyle{plain}
 \newtheorem*{cor*}{\protect\corollaryname}
 \newref{cor}{%
 name = \corollaryname~,
 lsttxt = \RSlsttxt,
 lsttwotxt = \RSlsttwotxt,
 refcmd = {\ref{#1}}}
 \theoremstyle{remark}
 \newtheorem*{rem*}{\protect\remarkname}
 \newref{rem}{%
 name = \remarkname~,
 lsttxt = \RSlsttxt,
 lsttwotxt = \RSlsttwotxt,
 refcmd = {\ref{#1}}}
 \def\SSEC[#1,#2,#3]{\renewcommand\subsection{\@startsection {subsection}{2}{\z@}%
 {#2ex \@plus -1ex \@minus -.2ex}%
 {#3ex \@plus.2ex}%
 {\reset@font #1}}}
 \newref{sec}{%
 name = \RSsectxt,
 names = \RSsecstxt,
 lsttxt = \RSlsttxt,
 lsttwotxt = \RSlsttwotxt,
 refcmd = \ref{#1}}
 \theoremstyle{remark}
 \newtheorem{srrem}{\protect\remarkname}[section]
 \newref{rem}{%
 name = \remarkname~,
 lsttxt = \RSlsttxt,
 lsttwotxt = \RSlsttwotxt,
 refcmd = {\ref{#1}}}
 \theoremstyle{definition}
  \newtheorem{sexmexample}{\protect\examplename}[section]
  \newref{exa}{%
  name = \examplename~,
  lsttxt = \RSlsttxt,
  lsttwotxt = \RSlsttwotxt,
  refcmd = {\ref{#1}}}
 \newref{prop}{%
 name = \propositionname~,
 lsttxt = \RSlsttxt,
 lsttwotxt = \RSlsttwotxt,
 refcmd = {\ref{#1}}}
 \theoremstyle{plain}
 \newcounter{ConSP}[spprop]
 \newtheorem{ccorsp}[ConSP]{\protect\corollaryname}
 \newref{cor}{%
 name = \corollaryname~,
 lsttxt = \RSlsttxt,
 lsttwotxt = \RSlsttwotxt,
 refcmd = {\ref{#1}}}
 \theoremstyle{definition}
  \newtheorem*{example*}{\protect\examplename}
  \newref{exa}{%
  name = \examplename~,
  lsttxt = \RSlsttxt,
  lsttwotxt = \RSlsttwotxt,
  refcmd = {\ref{#1}}}
 \theoremstyle{plain}
 \newtheorem{sllem}{\protect\lemmaname}[section]
 \newref{lem}{%
 name = \lemmaname~,
 lsttxt = \RSlsttxt,
 lsttwotxt = \RSlsttwotxt,
 refcmd = {\ref{#1}}}
 \theoremstyle{plain}

 \def\SSSEC[#1,#2,#3]{\renewcommand\subsubsection{\@startsection {subsubsection}{3}{\z@}%
 {#2ex \@plus -1ex \@minus -.2ex}%
 {#3ex \@plus.2ex}%
 {\reset@font #1}}}
 \newref{sec}{%
 name = \RSsectxt,
 names = \RSsecstxt,
 lsttxt = \RSlsttxt,
 lsttwotxt = \RSlsttwotxt,
 refcmd = \ref{#1}}

\def\section{\@startsection{section}{1}%
  \z@{.7\linespacing\@plus\linespacing}{.5\linespacing}%
  {\normalfont\scshape\bfseries\centering}}
\renewcommand\subsection{\@startsection {subsection}{2}{\z@}%
{.7\linespacing\@plus\linespacing \@minus -.2ex}%
{1ex \@plus.2ex}%
{\reset@font \normalfont\bfseries\upshape}}
\renewcommand\subsubsection{\@startsection {subsubsection}{3}{\z@}%
{.7\linespacing\@plus\linespacing \@minus -.2ex}%
{.5ex \@plus.2ex}%
{\reset@font \normalfont\bfseries\slshape}}
\newref{hypo}{%
refcmd = {Condition \ref{#1}}}
\def\@secnumfont{\normalfont\bfseries}

\newcommand\mynobreakpar{\par\nobreak\@afterheading}

\makeatother

 \providecommand{\corollaryname}{Corollary}
 \providecommand{\examplename}{Example}
 \providecommand{\lemmaname}{Lemma}
 \providecommand{\propositionname}{Proposition}
 \providecommand{\remarkname}{Remark}
 \providecommand{\theoremname}{Theorem}

\begin{document}
\def\rightmark{ESTIMATES FOR WEIGHTED BERGMAN PROJECTIONS}
\def\leftmark{P. CHARPENTIER, Y. DUPAIN \& M. MOUNKAILA}

\title{Estimates for weighted Bergman projections\\
on pseudo-convex domains of finite type in $\mathbb{C}^{n}$}

\author{P. Charpentier, Y. Dupain \& M. Mounkaila}
\begin{abstract}
In this paper we investigate the regularity properties of weighted
Bergman projections for smoothly bounded pseudo-convex domains of
finite type in $\mathbb{C}^{n}$. The main result is obtained for
weights equal to a non-negative rational power of the absolute value
of a special defining function $\rho$ of the domain: we prove (weighted) Sobolev-$L^{p}$
and Lipschitz estimates for domains in $\mathbb{C}^{2}$ (or, more
generally, for domains having a Levi form of rank $\geq n-2$ and
for ``decoupled'' domains) and for convex domains. In particular,
for these defining functions, we generalize results obtained by A.
Bonami \& S. Grellier and D. C. Chang \& B. Q. Li. We also obtain a general (weighted)
Sobolev-$L^{2}$ estimate.
\end{abstract}

\keywords{pseudo-convex, finite type, Levi form locally diagonalizable, convex,
extremal basis, geometric separation, weighted Bergman projection,
$\overline{\partial}_{\varphi}$-Neumann problem}

\subjclass[2000]{32F17, 32T25, 32T40}

\address{P. Charpentier \& Y. Dupain, Universit\'e Bordeaux I, Institut de
Math\'ematiques de Bordeaux, 351, Cours de la Lib\'eration, 33405,
Talence, France}

\address{M. Mounkaila, Universit\'e Abdou Moumouni, Facult\'e des Sciences,
B.P. 10662, Niamey, Niger}

\email{P. Charpentier: philippe.charpentier@math.u-bordeaux1.fr}

\email{Y. Dupain: yves.dupain@math.u-bordeaux1.fr}

\email{M. Mounkaila: modi.mounkaila@yahoo.fr}

\maketitle

\section*{Introduction}

Let $\Omega$ be a bounded open set in $\mathbb{C}^{n}$. Let $\omega$
be a non-negative measurable function on $\Omega$ and $\lambda$
be the Lebesgue measure on $\mathbb{C}^{n}$. The function $\omega$ is called
an \emph{admissible weight} (or simply a \emph{weight}) for $\Omega$
if the set $A^{2}\left(\Omega,\omega d\lambda\right)$ of square integrable
holomorphic functions with respect to the measure $\omega d\lambda$
is a closed subspace of the Hilbert space $L^{2}\left(\Omega,\omega d\lambda\right)$
(see \cite{Pas90}). So, if $\omega$ is a weight on $\Omega$,
the weighted Bergman projection $P_{\omega}^{\Omega}$, i.e. the orthogonal
projection of $L^{2}\left(\Omega,\omega d\lambda\right)$ onto $A^{2}\left(\Omega,\omega d\lambda\right)$, is well-defined. 

The aim of this paper is to investigate Lipschitz and Sobolev $L^{p}$
regularities of $P_{\omega}^{\Omega}$ when $\Omega$ is smooth, pseudo-convex
and of finite type.

As far as we know, for general finite type domains, only two results
were previously known. First, in \cite{BG95}, A. Bonami \& S. Grellier
proved Lipschitz and Sobolev $L^{p}$ estimates for the weighted Bergman
projection $P_{\omega}^{\Omega}$ of a finite type domain in $\mathbb{C}^{2}$ when the weight
$\omega$ is a non-negative \emph{entire} power of the absolute value of a defining function
$\rho$ of the domain (i.e. $\omega=\left(-\rho\right)^q$, $q\in\mathbb{N}$).
Secondly, in \cite{CDC97}, D. C. Chang \& B. Q. Li extended these results to ``decoupled''
domains in $\mathbb{C}^{n}$.

The main results of the present paper extend, for special defining
functions $\rho$, those estimates of $P_{\omega}^{\Omega}$ in two directions.
First we extend the class of weights $\omega$ to non-negative \emph{rational} powers of
$\left|\rho\right|$ (i.e. $\omega=\left(-\rho\right)^r$, $r\in\mathbb{Q}$).
Second we extend the class of domains to include convex domains (of finite type).

Moreover, we obtain weighted $L^{2}$-Sobolev regularity of
$P_{\left(-\rho\right)^r}^{\Omega}$ for general pseudo-convex domains of finite type in
$\mathbb{C}^{n}$, $\rho$ being also a special defining function of $\Omega$.

\bigskip{}

In complex analysis the (weighted or not) Bergman projection plays
a fundamental role and its regularity has been
extensively studied.

A fundamental class of weights is the one introduced by L. H\"ormander
in \cite{Hormander-L2-estimates} in order to solve the $\overline{\partial}$-equation.
Let $\varphi$ be a pluri-subharmonic function defined in $\Omega$.
H\"ormander's theorem solves the so-called $\overline{\partial}_{\varphi}$-Neumann
problem associated to the weight $e^{-\varphi}$ proving the existence
of the Neumann operator $\mathcal{N}_{\varphi}$ inverting the complex
laplacian $\square_{\varphi}$. Recall that the Bergman projection $P_{e^{-\varphi}}^{\Omega}$
is closely related to $\mathcal{N}_{\varphi}$ by the formula $P_{e^{-\varphi}}^{\Omega}=\mathrm{Id}-\overline{\partial}_{\varphi}^{*}\mathcal{N}_{\varphi}\overline{\partial}_{\varphi}$.

For $\varphi=0$ many results have been obtained in this direction in various
function spaces. In particular, for ($L^{2}$) Sobolev regularity
there is a very large bibliography essentially based on J. J. Kohn's
work (see \cite{Str10} for a good general presentation). For other
spaces, a lot of sharp results were obtained by several authors,
but there are still basic open problems (see for example \cite{N-R-S-W-Bergman-dim-2,Chang-Nagel-Stein,BC00,McNeal-convexes-94,McNeal-Stein-Bergman,Cho-Bergman-96,McNeal-Stein-Szego,Cho-03-Bergman-comparable-Math-Anal-Appl,Charpentier-Dupain-Geometery-Finite-Type-Loc-Diag,Charpentier-Dupain-Szego-Barcelone,CD08}
and references therein).

For non-zero functions $\varphi$, the only general result, due to
J. J. Kohn (\cite{Kohn-defining-function}), gives $L^{2}$-Sobolev
estimates for the $\overline{\partial}_{\varphi}$-Neumann problem
for general smoothly bounded pseudo-convex domains with $\varphi=t\left|z\right|^{2}$,
where $t$ is big enough depending on the Sobolev scale. Recall that there exist
smoothly bounded pseudo-convex domains for which the (unweighted) Bergman
projection is not $L^{2}$-Sobolev regular (\cite{Christ-96}). However,
if $\Omega$ is of finite type, it is not difficult to see that, if
$\varphi$ is $\mathcal{C}^{\infty}$ on $\overline{\Omega}$ then
the weighted Bergman projection $P_{e^{-\varphi}}^{\Omega}$ has the
same $L^{2}$-Sobolev regularity than the unweighted one.

\medskip{}

For the Bergman projection $P_{\omega}^{\Omega}$ with a general (admissible)
weight $\omega$, very few results were obtained for finite type domains.
In addition to the results of A. Bonami \& S. Grellier and D. C. Chang \& B. Q. Li
cited before, sharp results were obtained for strictly pseudo-convex domains. In
the case of the unit ball of $\mathbb{C}^{n}$, for weights equal
to a power greater than $-1$ of $1-\left|z\right|^{2}$, the kernels
of these operators can be written explicitly (see \cite{Cha80,HP84,HP84b})
and then it is possible to obtain very precise estimates. Generalizations
of these results to strictly pseudo-convex domains have also been
done by several authors (see \cite{LR86,LRM87,LR88,Cum90}).

Even in dimension $1$, $L^{p}$ estimates for weighted
Bergman projections can be true only for $p=2$, and, in general, are
not easy to obtain, as shown in \cite{Zey11b,Zey11,Zeyb,Zey}.

\bigskip{}

The method used in this paper is completely different than that
used in A. Bonami \& S. Grellier or D. C. Chang \& B. Q. Li papers.
It is inspired by a well-known method introduced by F. Forelli and
W. Rudin (see \cite{FR75,rudin-unit-ball,Lig89}):
we look at $\Omega$ as a slice of a pseudo-convex domain
$\widetilde{\Omega}$ of finite type in $\mathbb{C}^{n+m}$ and try
to deduce estimates for weighted Bergman projections of $\Omega$
from estimates of the unweighted Bergman projection of $\widetilde{\Omega}$.

\medskip{}

The paper is organized as follows. In the first section we present
the main results on weighted Bergman projections. In \secref{A-Hartogs-domain-Omega-Tilde}
we define the domain $\widetilde{\Omega}$ and discuss its fundamental properties.
In \secref{Relations-between-operators-omega-tilde-omega}
we give the basic relations between the Bergman projection of $\widetilde{\Omega}$
and a weighted Bergman projection of $\Omega$ and prove the general
$L^{2}$-Sobolev results. In \secref{Sharp-estimates-of-weighted-Bergman-kernel},
we prove the $L_{s}^{p}$ and Lipschitz estimates given in \thmref{Sobolev-Lp_Lip_Bergman}
establishing sharp estimates of the kernels of weighted Bergman projections.

\section{Main results and methods\label{sec:Main-results-and-methods}}

For simplicity, we only state here the main result concerning the
Bergman projection for weights equal to a non-negative rational power of $\left|\rho\right|$
where $\rho$ is a particular defining function of $\Omega$. Detailed
results for other operators, other weights and for the Bergman kernel
will be given in the next sections.

If $k$ is a positive function on an open set $\Omega$ in $\mathbb{C}^{n}$
whose inverse is locally integrable, we denote by $P_{k}=P_{k}^{\Omega}$
the orthogonal projection of the Hilbert space $L^{2}\left(kd\lambda\right)$
onto the (closed \cite{Pas90}) subspace of holomorphic functions
(i.e. the Bergman projection associated to the weight $k$).
\begin{stthm}
\label{thm:Sobolev-Lp_Lip_Bergman}Let $\Omega$ be a smoothly bounded
pseudo-convex domain of finite type $M$ and $\rho$ a defining function
of $\Omega$. We assume that one of the two following conditions is
satisfied:
\begin{itemize}
\item $\Omega$ is a domain in $\mathbb{C}^{2}$ and $\rho$ is such that
there exists $s\in\left]0,1\right]$ such that $-\left(-\rho\right)^{s}$
is strictly pluri-subharmonic in $\Omega$ (\cite{DF77-Strict-Psh-Exhau-Func-Inventiones});
\item $\Omega$ is convex and, if $g$ is a gauge of $\Omega$, then
$\rho=g^{4}e^{1-\nicefrac{1}{g}}-1$.
\end{itemize}
Let $\omega=\left(-\rho\right)^{r}$ with $r$ a non-negative rational
number. Let us denote by $P_{\omega}^{\Omega}$ the weighted Bergman
projection of $\Omega$ associated to the Hilbert space $L^{2}\left(\omega d\lambda\right)$,
$d\lambda$ denoting the Lebesgue measure. Let us also denote by $\delta_{\partial\Omega}$
the distance to the boundary of $\Omega$.
\begin{enumerate}
\item Let $s\in\mathbb{N}$. Then, for $p\in\left]1,+\infty\right[$ and
$-1<\beta<p\left(r+1\right)-1$, $P_{\omega}^{\Omega}$ maps
the Sobolev space $L_{s}^{p}\left(\delta^{\beta}_{\partial\Omega}\right)$ continuously into itself. 
\item For $\alpha<\nicefrac{1}{M}$, $P_{\omega}^{\Omega}$ maps
the Lipschitz space $\Lambda_{\alpha}$ continuously into itself. 
\end{enumerate}
\end{stthm}

For $L^{2}$-Sobolev estimates, we have more general results. For
example, for weighted Bergman projections:
\begin{stthm}
\label{thm:L2-Sobolev-Est-General}Let $\Omega$ be any smoothly bounded
pseudo-convex domain of finite type in $\mathbb{C}^{n}$. Let $\rho$
be a smooth defining function of $\Omega$ such that there exists
$s\in\left]0,1\right]$ such that $-\left(-\rho\right)^{s}$ is strictly
pluri-subharmonic in $\Omega$. Let $r\in\mathbb{Q}_{+}$.
Let $\omega=\left(-\rho\right)^{r}$. Then the weighted Bergman projection
$P_{\omega}^{\Omega}$ associated to the Hilbert space $L^{2}\left(\omega d\lambda\right)$
maps the Sobolev space $L_{s}^{2}\left(\omega d\lambda\right)$
continuously into itself for all $s\in\mathbb{N}$.
\end{stthm}

The method we use is a modification of a well-known construction
of Forelli-Rudin. Such method has been used by several authors for
the same kind of studies (for example E. Ligocka and Y. Zeytuncu
\cite{Lig89,Zey}). To try to obtain weights which are not only integer powers
of the absolute value of the defining function of $\Omega$, we consider
a more general situation investigating the properties of a domain
$\widetilde{\Omega}$ defined in $\mathbb{C}^{n+m}$ by an equation
of the form $\rho(z)+h(w)<0$ where $\rho$ is a defining function
of $\Omega$ and $h$ a positive function. These properties are discussed
in \secref{A-Hartogs-domain-Omega-Tilde}.
\begin{rem*}
\quad\mynobreakpar
\begin{enumerate}
\item The restriction $r\in\mathbb{Q}_{*}^{+}$ in our results is due to the
method and we do not know if the theorems can be extended to the natural
scale of powers which is $\left]-1,+\infty\right[$.
\item Recall that the results of A. Bonami \& S. Grellier and D. C.
Chang \& B. Q. Li are valid for any defining function of $\Omega$. It seems that it is not easy
to extend our results to any defining function. We will return to this question in a future paper.

The only thing we can say in general (i.e. without the finite type hypothesis)
is that the set of weights for which the corresponding weighted Bergman projections satisfy a given
estimate is, in a certain sense, open. This can be seen using a special case
of \cite[Theorem 4.2]{Pas90} (which can be easily directly proved): let $\omega$ be an admissible
weight, $h$ be a real smooth function such that $0<c\leq\left|h\right|\leq C<1$,
and denotes $M_{h}$ the operator $g\mapsto hg$. Then
\[
P_{\left(1+h\right)\omega}^{\Omega}=P_{\omega}^{\Omega}+\sum_{k=1}^{\infty}\left(-1\right)^{k-1}\left[P_{\omega}^{\Omega}\circ M_{h}\right]^{k}\circ\left(\mathrm{Id}-P_{\omega}^{\Omega}\right),
\]
 where $\left[...\right]^{k}$ is the power for the composition operator,
the series converging in the operator norm of
$\mathcal{L}\left(L^{2}\left(\omega d\lambda\right)\right)$.\\
For example, if $P_{\omega}^{\Omega}$ maps
$L^{p}\left(\delta_{\partial\Omega}^{\beta}\right)$
continuously into itself with norm $K^{\omega}_{p,\beta}$, then, a simple calculation
shows immediately that the same is true for $P_{\left(1+h\right)\omega}^{\Omega}$
if $\left\Vert h\right\Vert _{\infty}\leq\frac{1}{\left(K^{\omega}_{p,\beta}+1\right)^{2}}$.
\end{enumerate}
\end{rem*}

\section{\label{sec:A-Hartogs-domain-Omega-Tilde}A Hartogs domain \texorpdfstring{$\widetilde{\Omega}$}{Omega-tilde}
in $\mathbb{C}^{n+m}$ based on a domain \texorpdfstring{$\Omega$}{Omega}
in $\mathbb{C}^{n}$}

For a given smoothly bounded pseudo-convex domain $\Omega$ in $\mathbb{C}^{n}$
with a defining function $\rho$, we consider a smooth non-negative
function $h$ defined in $\mathbb{C}^{m}$ such that 
\[
h(w)=0\Leftrightarrow w=0\mbox{ and }\lim_{\left|w\right|\rightarrow+\infty}h(w)=+\infty
\]
and we denote by $\widetilde{\Omega}$ the smooth bounded domain
\begin{equation}
\widetilde{\Omega}=\left\{ \left(z,w\right)\in\mathbb{C}^{n}\times\mathbb{C}^{m}\mbox{, s. t. }r(z,w)=\rho(z)+h(w)<0\right\} .\label{eq:Definition-Omega-tilde}
\end{equation}

Then, for particular functions $h$, there are very simple relations
between the standard Bergman kernel of $\widetilde{\Omega}$ and a
weighted Bergman kernel of $\Omega$, and between the unweighted $\overline{\partial}$-Neumann
problem on $\widetilde{\Omega}$ and a weighted $\overline{\partial}_{\varphi}$-Neumann
problem on $\Omega$ (see \secref{Relations-between-operators-omega-tilde-omega}).
The goal of this part is to find what conditions on $\rho$
and $h$ would provide enough properties on $\widetilde{\Omega}$
so that we can obtain sharp estimates on the Bergman projection (or on
the $\overline{\partial}$-Neumann problem of $\widetilde{\Omega}$)
or sufficiently precise information on the Bergman kernel of $\widetilde{\Omega}$
on $\left\{ w=0\right\} $.

More precisely, in this section we discuss the following questions:
suppose $\nabla h(w)\neq0$ if $w\neq0$; under what conditions on
$\rho$ and $h$ the domain $\widetilde{\Omega}$ is:
\begin{itemize}
\item pseudo-convex;
\item pseudo-convex of finite type if $\Omega$ is of finite type.
\item completely geometrically separated, in the sense of \cite{CD08} (at
every boundary point or at boundary points near$\left\{ w=0\right\} $)
if $\Omega$ is so.
\end{itemize}

To get these properties, quite strong conditions have to be imposed
to $\rho$ and $h$. To simplify the reading of the paper, we first
state these different conditions.

\subsection{\label{sec:Diederich-Fornaess-def-func}A special defining function
for $\Omega$}

As we will see in \secref{Pseudo-convexity-of-Omega-tilde}, even
with the function $h(w)=\left|w\right|^{2}$, $w\in\mathbb{C}$, and
for very simple domains like the unit ball, $\widetilde{\Omega}$
is not automatically pseudo-convex: this depends on the choice of
the defining function $\rho$ (\remref{Pseudoconvexity-rho-plus-modulus-w}).

Fortunately, a good choice can always be done using a celebrated theorem
of K. Diederich \& J. E. Forn\ae ss (\cite[Theorem 1]{DF77-Strict-Psh-Exhau-Func-Inventiones})
which proves that for any smooth bounded pseudo-convex domain $\Omega$
there exists $s_{\Omega}\in\left]0,1\right]$ such that, for $s\in\left]0,s_{\Omega}\right[$
there exists a smooth defining function $\rho$ of $\Omega$ such
that the function $-(-\rho)^{s}$ is strictly pluri-subharmonic in
$\Omega$. Then, to fix the notations:

\emph{Throughout this paper, for $s\in\left]0,s_{\Omega}\right[$, we
will denote by $\rho_{s}$ a defining function such that
\begin{equation}
-\left(-\rho_{s}\right)^{s}\mbox{ is strictly pluri-subharmonic in }\Omega.\label{eq:hypothesis-defining-func}
\end{equation}
} Of course such a function $\rho_{s}$ is not unique.
\begin{srrem}
\label{rem:-log-of--D-F-def-function-psh}
\quad\mynobreakpar 
\begin{enumerate}
\item The pluri-subharmonicity of $-(-\rho_{s})^{s}$ means
\[
i\partial\overline{\partial}\rho_{s}\geq i\frac{1-s}{\rho_{s}}\partial\rho_{s}\wedge\overline{\partial}\rho_{s}.
\]
Thus $i\partial\overline{\partial}\rho_{s}\geq\frac{i}{\rho_{s}}\partial\rho_{s}\wedge\overline{\partial}\rho_{s}$,
and, as this means that $-\log\left(-\rho_{s}\right)$ is pluri-subharmonic
in $\Omega$, and the $\overline{\partial}_{-r\log\left(-\rho_{s}\right)}$-Neumann
problem is well-defined for $r\geq0$.
\item Let $U$ be an open set in $\mathbb{C}^{n}$. Let $\rho$ be a $\mathcal{C}^{\infty}$
function on $U$ whose gradient does not vanish. Assume there exists
$z^{0}\in U$ such that $\rho\left(z^{0}\right)=0$. Thus, for $\varepsilon\in\mathbb{R}$
sufficiently small the set 
\[
\Xi=\left\{ z\in U\mbox{ such that }\rho(z)=\varepsilon\right\} 
\]
is a non-empty smooth hypersurface. Assume moreover that $i\partial\overline{\partial}\rho\geq i\frac{\mu}{\rho}\partial\rho\wedge\overline{\partial}\rho$,
$\mu\in\mathbb{R}$, on $U$. Then, at every point of $\Xi$, the
restriction of $i\partial\overline{\partial}\rho$ to the complex
tangent space of $\Xi$ is non-negative.
\end{enumerate}
\end{srrem}

\subsection{\label{sec:Hypothesis_on_function_h}Hypothesis on the function $h$}

Depending on the properties we want to have for $\widetilde{\Omega}$,
several conditions will be imposed on $h$. We define now five conditions
that we will use in the following sections.

Let $\Omega$ be a bounded smooth pseudo-convex domain in $\mathbb{C}^{n}$
and let $h$ be a smooth real function on $\mathbb{C}^{m}$.
\begin{enumerate}[{}]\renewcommand{\labelenumi}{\hspace{-\parindent}}
\def\theenumi{\Roman{enumi}}%}]
\item \textbf{\label{hypo:Hypothesis-I}}\hspace{-2\parindent}\textbf{Condition
I}: \par\noindent{\itshape$\rho=\rho_{s}$ is a defining function
of $\Omega$ satisfying (\ref{eq:hypothesis-defining-func}) (\secref{Diederich-Fornaess-def-func}).
$h$ is non-negative, $h(w)=0\Leftrightarrow w=0$ (and thus $\nabla h(0)=0$),
$\nabla h(w)\neq0$ if $w\neq0$, $\lim_{\left|w\right|\rightarrow+\infty}h(w)=+\infty$
and there exists $s'\in\left[0,s\right[$ such that $h^{s'}$ is pluri-subharmonic
(i.e. $i\partial\overline{\partial}h\geq i\frac{1-s'}{h}\partial h\wedge\overline{\partial}h$
which implies that $h$ is strictly pluri-subharmonic at every point
$w$ such that $\frac{\partial h}{\partial w_{i}}(w)\neq0$ for all
$i$).}\vspace{.5em}
\item \textbf{\label{hypo:Hypothesis-II}}\hspace{-2\parindent}\textbf{Condition
II}: \par\noindent{\itshape$h(w)=\sum_{i=1}^{p}h_{i}\left(w_{i}\right)$,
$w=\left(w_{1},\ldots,w_{p}\right)$, $w_{i}\in\mathbb{C}^{m_{i}}$,
the functions $h_{i}$ being non-negative smooth pluri-subharmonic
on $\mathbb{C}^{m_{i}}$, and satisfying:}

\begin{enumerate}
\item {\itshape for every $i$, $h_{i}\left(w_{i}\right)=0\Leftrightarrow w_{i}=0$,
$\nabla h_{i}\left(w_{i}\right)\neq0$ if $w_{i}\neq0$, $\lim_{\left|w\right|\rightarrow+\infty}h_{i}(w)=+\infty$;}
\item {\itshape $\log\left(h_{i}\right)$ is pluri-subharmonic (i.e. $i\partial\overline{\partial}h_{i}\geq i\frac{1}{h_{i}}\partial h_{i}\wedge\overline{\partial}h_{i}$)
and $h_{i}$ is strictly pluri-subharmonic outside the origin;}
\item {\itshape $h_{i}$ is of finite type $2q_{i}=\mathrm{typ}_{0}\left(h_{i}\right)$
at the origin (in the sense introduced at the beginning of \secref{Type-finiteness}).}\vspace{.5em}
\end{enumerate}
\item \textbf{\label{hypo:Hypothesis-III}}\hspace{-2\parindent}\textbf{Condition
III}: \par\noindent{\itshape$h$ satisfies \hyporef{Hypothesis-II}
with $m_{i}=1$ for every $i$ (thus the first part of the second
condition of \hyporef{Hypothesis-II} means $\triangle h_{i}\geq\frac{1}{h_{i}}\left|h_{i}'\right|^{2}$),
and, for each $i$ there exists a function $\alpha_{i}$, $\mathcal{C}^{\infty}$
in a neighborhood of the origin in $\mathbb{C}$, $\lim_{w_{i}\rightarrow0}\alpha_{i}\left(w_{i}\right)=0$,
such that $\frac{\partial h_{i}}{\partial w_{i}}=\alpha_{i}\frac{\partial^{2}h_{i}}{\partial w_{i}\partial\overline{w}_{i}}$
in that neighborhood.}\vspace{.5em}
\item \textbf{\label{hypo:Hypothesis-IV}}\hspace{-2\parindent}\textbf{Condition
IV}: \par\noindent{\itshape$\Omega$ is of finite type, $h$ satisfies
\hyporef{Hypothesis-III}, and, for each $i$, $h_{i}\left(w_{i}\right)=k_{i}\left(\left|w_{i}\right|\right)$,
$k_{i}(t)\asymp t^{2q_{i}}$ (where $f\asymp g$ means that there
exist two constants $c>0$ and $C>0$ such that $cf\leq g\leq Cf$)
and $2q_{i}$ strictly larger than the type of $\Omega$.}\vspace{.5em}
\item \textbf{\label{hypo:Hypothesis-V}}\hspace{-2\parindent}\textbf{Condition
V}: \par\noindent{\itshape$\Omega$ is of finite type, $h$ satisfies
\hyporef{Hypothesis-II} with $m_{i}=1$ for every $i$, and, for
each $i$, $\left|\frac{\partial^{2}h_{i}}{\partial w_{i}^{2}}\right|\leq\triangle h_{i}$,
$h_{i}\left(w_{i}\right)=k_{i}\left(\left|w_{i}\right|\right)$ and
$k_{i}(t)\asymp t^{2q_{i}}$ and $q_{i}$ strictly larger than the
type of $\Omega$.}
\end{enumerate}

\hyporef{Hypothesis-I} is used in \secref{Pseudo-convexity-of-Omega-tilde}
to get the pseudo-convexity of $\widetilde{\Omega}$. \hyporef{Hypothesis-II}
is used in \secref{Type-finiteness} to ensure that $\widetilde{\Omega}$
is of finite type. In \secref{Geometric-separation} we use Condition
\ref{hypo:Hypothesis-III} to obtain that $\widetilde{\Omega}$ have
a Levi form locally diagonalizable (and thus is ``completely geometrically
separated'' (\cite{CD08})) when $\Omega$ is a finite type domain
in $\mathbb{C}^{2}$. Finally, Conditions \ref{hypo:Hypothesis-IV}
and \ref{hypo:Hypothesis-V} are used in Sections \ref{sec:C2-case}
and \ref{sec:convex-case} to get pointwise estimates of a weighted
Bergman kernel in these two cases.
\begin{sexmexample}
\label{exa:Basic_example_function_h}Let $m_{i}\in\mathbb{N}_{*}$,
$1\leq i\leq p$, $m=\sum_{i}m_{i}$, $q_{i}\in\mathbb{N}_{*}$, $1\leq i\leq p$.
Then the function
\[
h:\, w=\left(w_{1},\ldots w_{p}\right)\in\prod_{i}\mathbb{C}^{m_{i}}=\mathbb{C}^{m}\mapsto\sum_{i}\left|w_{i}\right|^{2q_{i}}
\]
satisfies Conditions \ref{hypo:Hypothesis-I} (whatever $s$) and
\ref{hypo:Hypothesis-II}, and all other conditions if $m_{i}=1$
for every $i$ and the $q_{i}$ are large enough.\end{sexmexample}
\begin{proof}
Let us denote $h_{i}\left(w_{i}\right)=\left|w_{i}\right|^{2q_{i}}$,
$1\leq i\leq p$. Then $\overline{\partial}h=\sum_{i}\overline{\partial}h_{i}$,
$\partial\overline{\partial}h=\sum_{i}\partial\overline{\partial}h_{i}$
and, a simple calculation and Cauchy-Schwarz inequality give, for $1\leq i\leq p$,
\begin{eqnarray*}
i\partial\overline{\partial}h_{i} & = & iq_{i}^{2}\left|w_{i}\right|^{2q_{i}-4}\left(\sum_{j}\overline{w_{i}^{j}}dw_{i}^{j}\right)\wedge\left(\sum_{j}w_{i}^{j}d\overline{w_{i}^{j}}\right)+\\
 &  & +iq_{i}\left|w_{i}\right|^{2q_{i}-2}\left(\sum_{j}dw_{i}^{j}\wedge d\overline{w_{i}^{j}}\right)-iq_{i}\left|w_{i}\right|^{2q_{i}-4}\left(\sum_{j}\overline{w_{i}^{j}}dw_{i}^{j}\right)\wedge\left(\sum_{j}w_{i}^{j}d\overline{w_{i}^{j}}\right)\\
 & \geq & i\frac{1}{h_{i}}\partial h_{i}\wedge\overline{\partial}h_{i}.
\end{eqnarray*}
Then, by Cauchy-Schwarz inequality, 
\[
ih\left\langle \partial\overline{\partial}h;t,\overline{t}\right\rangle \geq\left(\sum_{i}h_{i}\right)\left(\sum\frac{i\left\langle \partial h_{i}\wedge\overline{\partial}h_{i};t_{i},\overline{t_{i}}\right\rangle }{h_{i}}\right)\geq\left(\sum\left\langle \partial h_{i}\wedge\overline{\partial}h_{i};t_{i},\overline{t_{i}}\right\rangle ^{\nicefrac{1}{2}}\right)^{2}\geq i\left\langle \partial h\wedge\overline{\partial}h;t,\overline{t}\right\rangle 
\]
and \hyporef{Hypothesis-I} is satisfied for any $s'\geq0$.
\end{proof}

\subsection{\label{sec:Pseudo-convexity-of-Omega-tilde}Pseudo-convexity of \texorpdfstring{$\widetilde{\Omega}$}{Omega-tilde}}

In general, even for very simple functions $h$, the domain $\widetilde{\Omega}$
is not pseudo-convex. For example, if $\Omega$ is the unit ball it
is very easy to write a defining function $\rho$ of $\Omega$ such
that the domain $\left\{ \left(z,w\right)\in\mathbb{C}^{n}\times\mathbb{C}\mbox{, s. t. }\rho(z)+\left|w\right|^{2}<0\right\} $
is not pseudo-convex (see \remref{Pseudoconvexity-rho-plus-modulus-w}).

If $\Omega$ admits a smooth defining function which is pluri-subharmonic
in $\overline{\Omega}$, it suffices to take $h$ pluri-subharmonic.
But this is not the general case (c.f. \cite{DF77-nebenhuelle}),
so, we have to choose a convenient defining function for $\Omega$:
\begin{spprop}
\label{prop:Local-pseudoconvexity-for-rho-plus-h}Let $Z\in\mathbb{C}^{n}$
and let $U$ (resp. $V$) be an open neighborhood of $Z$ (resp. of
the origin in $\mathbb{C}^{m}$). Let $\rho\,:U\rightarrow\mathbb{R}$
(resp. $h\,:V\rightarrow\mathbb{R}_{+}$) be a smooth function such
that $\rho\left(Z\right)=0$ and $\nabla\rho$ does not vanishes in
$U$ (resp $\nabla h(w)\neq0$ if and only if $w\neq0$ and $h(w)\neq0$
if $w\neq0$). Assume that there exist $s\in\left]0,1\right]$ and
$s'\in\left[0,s\right[$ such that $-(-\rho)^{s}$ is strictly pluri-subharmonic
in the open set $G=\left\{ z\in U\mbox{ s. t. }\rho(z)<0\right\} $
and $h^{s'}$ is pluri-subharmonic in $V$. Let $\partial\widetilde{G}$
be the set
\[
\partial\widetilde{G}=\left\{ \left(z,w\right)\in U\times V\mbox{ s. t. }r(z,w)=\rho(z)+h(w)=0\right\} .
\]
Then, if not empty $\partial\widetilde{G}$ is a smooth hypersurface,
and, at every point $\left(z^{1},w^{1}\right)\in\partial\widetilde{G}$:
\begin{enumerate}
\item The restriction of $i\partial\overline{\partial}r$ to the complex
tangent space of $\partial\widetilde{G}$ at $\left(z^{1},w^{1}\right)$
is non-negative.
\item If $w^{1}\neq0$, the restriction of $i\partial\overline{\partial}r$
to the complex tangent space of $\partial\widetilde{G}$ at $\left(z^{1},w^{1}\right)$
is positive definite if $h$ is strictly pluri-subharmonic at $w^{1}$.
\end{enumerate}
\end{spprop}
\begin{proof}
Note that, since the gradient of $r$ does not vanishes on $U\times V$,
$\partial\widetilde{G}$ is a smooth hypersurface if it is not empty. Moreover,
$s'$ being $\leq1$, $h$ is pluri-subharmonic on $V$. The hypothesis on $\rho$ is
\begin{equation}
i\partial\overline{\partial}\rho\geq i\frac{1-s}{\rho}\partial\rho\wedge\overline{\partial}\rho+\varepsilon i\partial\overline{\partial}\left|z\right|^{2},\label{eq:Prop-D-F-def-func}
\end{equation}
where $\varepsilon$ is a non-negative function, strictly positive
in $G$.

Let $\left(z^{1},w^{1}\right)$ be a point of $\partial\widetilde{G}$
and $t=\left(t_{z},t_{w}\right)\in\mathbb{C}^{n}\times\mathbb{C}^{m}$
be a vector of the complex tangent space of $\partial\widetilde{G}$
at $\left(z^{1},w^{1}\right)$.

If $w^{1}=0$, as the hypotheses on $h$ imply $\nabla h(0)=0$, $t_{z}$
is ``tangent'' to $\rho$ at $z^{1}$ (i.e. $\left\langle \partial\rho,t_{z}\right\rangle\left(z^1\right) =0$).
Thus the inequality (\ref{eq:Prop-D-F-def-func}) and the pluri-subharmonicity of $h$ imply
$\left\langle i\partial\overline{\partial}r;t,\overline{t}\right\rangle \geq0$
showing (1).

Suppose now $w^{1}\neq0$. Note that, $t$ being tangent to $r$,
we have $\sum_{i}\frac{\partial\rho}{\partial z_{i}}\left(z^{1}\right)t_{z}^{i}=-\sum_{j}\frac{\partial h}{\partial w_{j}}\left(w^{1}\right)t_{w}^{j}$,
and $\rho\left(z^{1}\right)=-h\left(w^{1}\right)$. Then (\ref{eq:Prop-D-F-def-func})
gives
\[
\left\langle i\partial\overline{\partial}\rho\left(z^{1}\right);t_{z},\overline{t_{z}}\right\rangle \geq\frac{1-s}{-h\left(w^{1}\right)}\left|\sum_{j}\frac{\partial h}{\partial w_{j}}\left(w^{1}\right)t_{w}^{j}\right|^{2}+\varepsilon\left|t_{z}\right|^{2}.
\]

Thus
\[
\left\langle i\partial\overline{\partial}r\left(z^{1},w^{1}\right);t,\overline{t}\right\rangle \geq\frac{1-s}{-h\left(w^{1}\right)}\left|\sum_{j}\frac{\partial h}{\partial w_{j}}\left(w^{1}\right)t_{w}^{j}\right|^{2}+\left\langle i\partial\overline{\partial}h\left(w^{1}\right);t_{w},\overline{t_{w}}\right\rangle +\varepsilon\left|t_{z}\right|^{2}.
\]

The conclusion comes then from the fact that the hypothesis made on
$h$ is
\[
i\partial\overline{\partial}h\geq\max\left\{ i\frac{1-s'}{h}\partial h\wedge\overline{\partial}h,\varepsilon_{1}i\partial\overline{\partial}\left|w\right|^{2}\right\} ,
\]
with $0\leq s'<s$ and $\varepsilon_{1}$ a non-negative function,
which is positive at $w^{1}$ if $h$ is strictly pluri-subharmonic
at that point: $s'\leq s$ gives immediately (1), and, if $h$ is strictly pluri-subharmonic
at $w^{1}$, then $\varepsilon\left(z^{1}\right)>0$, $\varepsilon_{1}\left(w^{1}\right)>0$
and
\[
\left\langle i\partial\overline{\partial}r\left(z^{1},w^{1}\right);t,\overline{t}\right\rangle \geq\varepsilon_{2}\left|t\right|^{2},
\]
for $\varepsilon_{2}>0$ small.\end{proof}
\begin{spprop}
\label{prop:Global_pseudoconvexity-for-rho-plus-h}Let $\Omega$ be
a bounded pseudo-convex domain, with smooth boundary. Let $h$ be
a smooth non-negative function defined in $\mathbb{C}^{m}$. Let $\rho=\rho_{s}$
be a smooth defining function of $\Omega$ satisfying (\ref{eq:hypothesis-defining-func})
as stated in \secref{Diederich-Fornaess-def-func}.
\begin{enumerate}
\item \label{1-of-Global-pseudoconvexity-rho-plus-h}Then, if $h$ satisfies
\hyporef{Hypothesis-I} of \secref{Hypothesis_on_function_h}, $\widetilde{\Omega}$
is pseudo-convex.
\item \label{2-of-Global-pseudoconvexity-rho-plus-h}Moreover, for $\left(z^{0},w^{0}\right)\in\partial\widetilde{\Omega}$,
$w^{0}\neq0$ if, in addition, $h$ is strictly pluri-subharmonic
at $w^{0}$, $\widetilde{\Omega}$ is strictly pseudo-convex at $\left(z^{0},w^{0}\right)$.
\end{enumerate}
\end{spprop}
\begin{proof}
Let $\left(z^{0},w^{0}\right)\in\partial\widetilde{\Omega}$. If $\partial\rho\left(z^{0}\right)\neq0$,
$\widetilde{\Omega}$ is pseudo-convex at $\left(z^{0},w^{0}\right)$
by \propref{Local-pseudoconvexity-for-rho-plus-h}. If $\partial\rho\left(z^{0}\right)=0$,
then $z^{0}\in\Omega$ and (\ref{eq:Prop-D-F-def-func}) shows that
$\rho$ is strictly pluri-subharmonic in a neighborhood of $z^{0}$.
Thus $r$ is pluri-subharmonic in a neighborhood of $\left(z^{0},w^{0}\right)$
and strictly pluri-subharmonic if $h$ is strictly pluri-subharmonic
at $w^{0}$.\end{proof}
\begin{srrem}
\label{rem:Pseudoconvexity-rho-plus-modulus-w}When $m=1$ and $h\left(w\right)=\left|w\right|^{2}$,
the domain $\widetilde{\Omega}=\left\{ \rho(z)+h(w)<0\right\} $ is pseudo-convex
at a point $\left(z^{1},w^{1}\right)\in\partial\widetilde{\Omega}$, $w^{1}\neq0$, if and only
if
\[
i\partial\overline{\partial}\rho\left(z^{1}\right)\geq\frac{i}{\rho\left(z^{1}\right)}\left(\partial\rho\wedge\overline{\partial}\rho\right)\left(z^{1}\right).
\]

For example, if $\rho$ is the signed distance to the boundary of
$\Omega$, by Oka's theorem, $\left\{ \rho(z)+\left|w\right|^{2}<0\right\} $
is pseudo-convex.\end{srrem}
\begin{sexmexample}
With the function $h$ given in \exaref{Basic_example_function_h},
(\ref{1-of-Global-pseudoconvexity-rho-plus-h}) of \propref{Global_pseudoconvexity-for-rho-plus-h}
applies for any pseudo-convex domain $\Omega$. Note also that, $\widetilde{h_{1}}(w)=h(w)+\left|w\right|^{2q}$,
$q\in\mathbb{N}_{*}$ satisfies the condition stated in (\ref{2-of-Global-pseudoconvexity-rho-plus-h})
of the proposition.\end{sexmexample}
\begin{proof}
As noted in the proof of \exaref{Basic_example_function_h}, $i\partial\overline{\partial}\left(\left|w\right|^{2q}\right)\geq\frac{i}{\left|w\right|^{2q}}\partial\left(\left|w\right|^{2q}\right)\wedge\overline{\partial}\left(\left|w\right|^{2q}\right)$.\end{proof}
\begin{srrem}
For $\Omega$ pseudo-convex in $\mathbb{C}^{n}$, E. Ligocka considered,
in \cite{Lig89}, the domains 
\[
\widetilde{\Omega}_{N,k}=\left\{ \left(z,w\right)\in\mathbb{C}^{n}\times\mathbb{C}^{m}\mbox{ s. t. }\rho(z)+\left|w\right|^{2Nk}<0\right\} ,
\]
where $k$ is a sufficiently large integer such that the defining
function $\rho$ of $\Omega$ satisfies that $-\left(-\rho\right)^{\nicefrac{1}{k}}$
is strictly plurisubharmonic in $\Omega$ and $N$ a positive integer.
She showed that if $\Omega$ is ``weakly regular'', then so is $\widetilde{\Omega}_{N,k}$.
\end{srrem}

\subsection{\label{sec:Type-finiteness}Type finiteness of \texorpdfstring{$\widetilde{\Omega}$}{Omega-tilde}}

We now investigate the question concerning the type of the domain
$\widetilde{\Omega}$ defined by \eqref{Definition-Omega-tilde} with
a defining function $\rho$ of $\Omega$ \emph{which does not necessarily
satisfy (\ref{eq:hypothesis-defining-func})}.

Let us introduce first a notation. If $g$ is a real or complex valued
smooth function defined in a neighborhood of the origin in $\mathbb{R}^{d}$,
we call the \emph{order} of $g$ at the origin the integer $\mbox{ord}_{0}(g)$
defined by $\mbox{ord}_{0}(g)=\infty$ if $g^{(\alpha)}(0)=0$ for
all multi-index $\alpha\in\mathbb{N}^{d}$ and 
\[
\mbox{ord}_{0}(g)=\min\left\{ k\in\mathbb{N}\mbox{ such that there exists }\alpha\in\mathbb{N}^{d},\,\left|\alpha\right|=\sum\alpha_{i}=k\mbox{ such that }g^{(\alpha)}(0)\neq0\right\} 
\]
otherwise. Moreover if $f=\left(f_1,\ldots,f_q\right)$ is a smooth function defined in a neighborhood
of the origin in $\mathbb{R}^{d}$ with values in $\mathbb{R}^q$ the \emph{order} of $f$ at the origin is
$\mbox{ord}_{0}(f)=\min\left\{ \mbox{ord}_{0}\left(f_{i}\right),\,1\leq i\leq q\right\} $.
If $h$ is a smooth function defined in a neighborhood
of the origin in $\mathbb{C}^{m}$, then, for all function $\varphi$
from the unit disc of the complex plane into $\mathbb{C}^{m}$ such
that $\varphi(0)=0$, $h\circ\varphi$ is smooth in a neighborhood
of the origin in $\mathbb{C}$. Then we call the \emph{type} of $h$
at the origin the supremum of
$\frac{\mbox{ord}_{0}\left(h\circ\varphi\right)}{\mbox{ord}_{0}\left(\varphi\right)}$,
taken over all non-zero holomorphic function $\varphi$ from the unit
disc of the complex plane into $\mathbb{C}^{m}$ such that $\varphi(0)=0$.
If this supremum is finite, we will say that $h$ is of \emph{finite
type} at the origin and we will denote this supremum by $\mbox{typ}_{0}(h)$.
Moreover, if $k$ is a smooth function defined in a neighborhood of
a point $z_{0}\in\mathbb{C}^{m}$, the type $\mbox{typ}_{z_{0}}\left(k\right)$
of $k$ at $z_{0}$ is $\mbox{typ}_{0}\left(h_{k}\right)$ where $h_{k}(z)=k\left(z_{0}+z\right)$
and we say that $k$ is of finite type at $z_{0}$ if $\mbox{typ}_{z_{0}}\left(k\right)<+\infty$.
\begin{spprop}
\label{prop:Local-rho+h-negative-finite-type}Let $z_{0}\in\mathbb{C}^{n}$
and $U$ (resp. $V$) be an open neighborhood of $z_{0}$ (resp. of
the origin in $\mathbb{C}^{m}$). Let $\rho\,:U\rightarrow\mathbb{R}$
(resp. $h\,:V\rightarrow\mathbb{R}_{+}$) be a smooth function such
that $\rho\left(z_{0}\right)=0$ and $\nabla\rho$ does not vanishes
in $U$ (resp $\nabla h(w)\neq0$ if $w\neq0$ and $h(w)\neq0$ if
and only if $w\neq0$). Assume that the restriction of $i\partial\overline{\partial}\rho$
to the complex tangent space to the hypersurface $\partial G=\left\{ z\in U\mbox{ s. t. }\rho(z)=0\right\} $
is non-negative, that $\partial G$ is of finite type $\tau$ at the
point $z_{0}$ and that $h$ is of finite type $\mathrm{typ}_{0}(h)$
at the origin.

Then the boundary of $\widetilde{G}=\left\{ \left(z,w\right)\in U\times V\mbox{ s. t. }r(z,w)=\rho(z)+h(w)<0\right\} $
is of finite type $\max\left(\tau,\mathrm{typ}_{0}(h)\right)$ at
the point $\left(z_{0},0\right)$.\end{spprop}
\begin{proof}
To simplify the notations, we can assume $z_{0}=0$. Let $\Phi=\left(\varphi,\psi\right)$
be a non-constant holomorphic function from the unit disc of the complex
plane into $U\times V$ such that $\Phi(0)=0$. We want to estimate
$\frac{\mbox{ord}_{0}(r\circ\Phi)}{\mbox{ord}_{0}(\Phi)}$, and therefore
we can assume that the first derivative of $r\circ\Phi$ vanishes
at the origin. As the gradient of $h$ vanishes at the origin, this
implies that the first derivative of $\rho\circ\varphi$ vanishes
at the origin which means that the gradient of $\varphi$ at $0$
is ``tangent'' to $\rho$ at $0$ (i.e.
$\sum\frac{\partial\varphi_{i}}{\partial\zeta}\frac{\partial\rho}{\partial z_{i}}(0)=0$).
Clearly, by the hypothesis made on $\partial G$, we can assume that the order
of $h\circ\psi$ at the origin is not
infinite, and then, by Lemma 8.1 of \cite{CD08} and the hypothesis made on $h$,
there exists an integer $k\geq1$ such that all the derivatives of
order $<2k$ of $h\circ\psi$ vanish at the origin and $\triangle^{k}\left(h\circ\psi\right)(0)>0$.

As $\mbox{ord}_{0}\left(\Phi\right)$ is the minimum of $\mbox{ord}_{0}\left(\varphi\right)$
and $\mbox{ord}_{0}\left(\psi\right)$, to prove the proposition it
suffices to prove that
\[
\mbox{ord}_{0}\left(r\circ\Phi\right)\leq\min\left\{ \mbox{ord}_{0}\left(\rho\circ\varphi\right),\mbox{ord}_{0}\left(h\circ\psi\right)\right\} =\min\left\{ \mbox{ord}_{0}\left(\rho\circ\varphi\right),2k\right\} 
\]
(note that this not a trivial consequence of the fact that $\rho$
and $h$ are decoupled).

Note first that, if $\mbox{ord}_{0}\left(\rho\circ\varphi\right)<2k$
or $\mbox{ord}_{0}\left(\rho\circ\varphi\right)>2k$, the inequality
is obvious. Thus we can assume $\mbox{ord}_{0}\left(\rho\circ\varphi\right)=2k$,
and we have to prove that $\mbox{ord}_{0}\left(r\circ\Phi\right)=2k$.

Suppose it is not the case: it follows that all the derivatives
of $r\circ\Phi$ of order $\leq2k$ vanish at $0$. Consider $\triangle^{k}\left(\rho\circ\varphi\right)(0)$.
As $\varphi$ is holomorphic, we have
\begin{equation}
\triangle^{k}\left(\rho\circ\varphi\right)(0)=\triangle^{k-1}\left(\left\langle \partial\overline{\partial}\rho(\varphi);\varphi',\overline{\varphi'}\right\rangle \right)(0).\label{eq:laplacian-to-k-rho-circle-phi}
\end{equation}
As all the derivatives of $\rho\circ\varphi$ of order less or equal
than $2k-1$ vanish at the origin, we have $\rho\circ\varphi\left(\zeta\right)=\mbox{O}\left(\zeta^{2k}\right)$.
For $\zeta$ small, let $\xi=\xi(\zeta)$ be the projection of $\varphi(\zeta)$
on $\left\{ z\in U\mbox{ s. t. }\rho(z)=0\right\} $ so that $\varphi(\zeta)-\xi=\mbox{O}\left(\zeta^{2k}\right)$.
By the hypothesis $\mbox{ord}_{0}\left(\rho\circ\varphi\right)=2k$,
all the derivatives of $\left\langle \partial\rho(\varphi);\varphi'\right\rangle $
of order less or equal than $2k-2$ vanish at the origin, which implies
that there exists a vector $T=T(\zeta)$ tangent to $\left\{ \rho=\rho(\xi)\right\} $ at the
point $\xi$ such that $\varphi'(\zeta)=T+\mbox{O}\left(\zeta^{2k-1}\right)$.
Then, by hypothesis on the hypersurface $\partial G$, we have
\[
\left\langle i\partial\overline{\partial}\rho(\varphi(\zeta));\varphi'(\zeta),\overline{\varphi'(\zeta)}\right\rangle =\left\langle i\partial\overline{\partial}\rho(\xi);T,\overline{T}\right\rangle +\mbox{O}\left(\zeta^{2k-1}\right)\geq-C\left|\zeta\right|^{2k-1}.
\]

Applying Lemma 8.1 of \cite{CD08} to the positive function 
\[
\left\langle i\partial\overline{\partial}\rho(\varphi(\zeta));\varphi'(\zeta),\overline{\varphi'(\zeta)}\right\rangle +C\left|\zeta\right|^{2k-1},
\]
by (\ref{eq:laplacian-to-k-rho-circle-phi}), we get $\triangle^{k}\left(\rho\circ\varphi\right)(0)\geq0$
which is impossible since $0=\triangle^{k}\left(r\circ\Phi\right)(0)=\triangle^{k}\left(\rho\circ\varphi\right)(0)+\triangle^{k}\left(h\circ\psi\right)(0)>0$.
\end{proof}

Recall that in \propref{Local-rho+h-negative-finite-type} we do not
assume that $\rho$ necessarily satisfies (\ref{eq:hypothesis-defining-func}).

Applying this proposition to a general pseudo-convex domain, we get:
\begin{ccorsp}
\label{cor:Global_finite-type_rho+h}Assume that $\rho=\rho_{s}$
where $\rho_{s}$ satisfy (\ref{eq:hypothesis-defining-func}) as
stated in \secref{Diederich-Fornaess-def-func}.
\begin{enumerate}
\item \label{1-Cor-global-rho-plus-h-finite-type}Assume that $h$ satisfies
\hyporef{Hypothesis-I} of \secref{Hypothesis_on_function_h}. Let
$z^{0}$ be a boundary point of $\Omega$. If $\partial\Omega$ is
of finite type $\tau$ at $z^{0}$ and if $h$ is of finite type $\mathrm{typ}_{0}(h)$
at the origin, then $\partial\widetilde{\Omega}$ is of finite type
$\max\left(\tau,\mathrm{{typ}}_{0}(h)\right)$ at the point $\left(z^{0},0\right)$.
Moreover, if $h$ is strictly pluri-subharmonic in $\mathbb{C}^{n}\setminus\left\{ 0\right\} $
then $\partial\widetilde{\Omega}$ is strictly pseudo-convex at every
boundary point $\left(z^{0},w^{0}\right)$ such that $w^{0}\neq0$.
\item \label{2-Cor-global-rho-plus-h-finite-type}Assume that $h$ satisfy
\hyporef{Hypothesis-II} of \secref{Hypothesis_on_function_h}. Then
$\widetilde{\Omega}$ is pseudo-convex and, at every point $\left(z^{0},w^{0}\right)\in\partial\widetilde{\Omega}$,
$w^{0}\neq0$, $\partial\widetilde{\Omega}$ is of finite type $\max_{i\mbox{ s. t. }w_{i}^{0}=0}\left\{ \mathrm{{typ}}_{0}\left(h_{i}\right)\right\} $
if there exists some $i$ such that $w_{i}^{0}=0$ and is strictly pseudo-convex
if $w_{i}^{0}\neq0$ for all $i$.
\end{enumerate}
\end{ccorsp}
\begin{proof}
The first part of (\ref{1-Cor-global-rho-plus-h-finite-type}) is
a special case of \propref{Local-rho+h-negative-finite-type}, and
the second part is stated in \propref{Global_pseudoconvexity-for-rho-plus-h}.

Let us now prove (\ref{2-Cor-global-rho-plus-h-finite-type}). The
pseudo-convexity of $\widetilde{\Omega}$ follows from the results of \secref{Pseudo-convexity-of-Omega-tilde}.
If $w_{i}^{0}\neq0$ for all $i$, as before, $\partial\widetilde{\Omega}$
is strictly pseudo-convex at $\left(z^{0},w^{0}\right)$. So, assume
that there exists some $i$ such that $w_{i}^{0}=0$. Without loss
of generality, we can suppose that $w_{k+1}^{0}=\ldots=w_{p}^{0}=0$,
$k<p$, and $w_{l}^{0}\neq0$ for $1\leq l\leq k$. Let us denote
$w=\left(w',w''\right)$, with $w'=\left(w_{1},\ldots,w_{k}\right)$,
$w''=\left(w_{k+1},\ldots,w_{p}\right)$, and $\rho_{1}\left(z,w'\right)=\rho(z)+\sum_{i=1}^{k}h_{i}\left(w_{i}\right)$.
In a neighborhood of $\left(z^{0},w'^{0}\right)$, $\nabla\rho_{1}$
does not vanish, $\left\{ \rho_{1}<0\right\} $ is strictly pseudo-convex,
and we can apply \propref{Local-rho+h-negative-finite-type} to the
domain $\rho_{1}\left(z,w'\right)+\sum_{i=k+1}^{p}h_{i}\left(w_{i}\right)<0$
and the function $h_{1}\left(w''\right)$ at the point $\left(z^{0},w'^{0},0\right)$.
\end{proof}

When $\Omega$ admits a pluri-subharmonic defining function which
is of finite type (in the sense defined at the beginning of the section)
everywhere in $\overline\Omega$, the proposition gives:
\begin{ccorsp}
\label{cor:Global-finite-type-psh-def-func}Assume $\Omega$ admits
a defining function $\rho$ pluri-subharmonic in a neighborhood of
$\overline{\Omega}$ and of finite type in $\overline{\Omega}$. Assume
that $h$ satisfies \hyporef{Hypothesis-II} of \secref{Hypothesis_on_function_h}.
Then the domain $\widetilde{\Omega}$, defined with $\rho$, is pseudo-convex
of finite type. More precisely, at every point $\left(z^{1},w^{1}\right)\in\partial\widetilde{\Omega}$
the type of $\partial\widetilde{\Omega}$ is bounded by $\max\left\{ {\mathrm{typ}}_{z^{1}}(\partial\Omega),{\mathrm{typ}}_{0}(h_{i}),\,1\leq i\leq p\right\} $
if $w^{1}=0$ and by 
\[
\max\left\{ 2\mathrm{typ}_{z^{1}}(\rho),\,\max_{i\mbox{ such that }w_{i}\neq0}\left\{ \mathrm{typ}_{0}\left(h_{i}\right)\right\} \right\} 
\]
 otherwise.\end{ccorsp}
\begin{proof}
Let us first consider the case $p=1$. By \propref{Local-rho+h-negative-finite-type},
$\widetilde{\Omega}$ is of finite type $\max\left\{ {\mathrm{typ}}_{z^{1}}(\partial\Omega),{\mathrm{typ}}_{0}(h)\right\} $
at every point $\left(z^{1},0\right)\in\partial\widetilde{\Omega}$,
and we have to study the finiteness at points $\left(z^{1},w^{1}\right)\in\partial\widetilde{\Omega}$
such that $w^{1}\neq0$. Let $\Phi=\left(\varphi,\psi\right)$ be
a non-constant holomorphic function from the unit disc of the complex
plane into a neighborhood of $\left(z^{1},w^{1}\right)$ such that
$\Phi(0)=\left(z^{1},w^{1}\right)$. We have to estimate $\tau_{\Phi}=\frac{\mbox{ord}_{0}\left(r\circ\Phi\right)}{\mbox{ord}_{0}(\Phi-\Phi(0))}$,
and we can, of course, assume that $\mathrm{ord}_{0}\left(r\circ\Phi\right)\geq2$.

Let $k$ be a positive integer and let us assume that all the derivatives
of order $\leq2k$ of $r\circ\Phi$ vanish at the origin. Then
\[
\triangle\left(r\circ\Phi\right)(0)=\left\langle \partial\overline{\partial}\rho;\varphi',\overline{\varphi'}\right\rangle (0)+\left\langle \partial\overline{\partial}h;\psi',\overline{\psi'}\right\rangle (0)=0,
\]
and the hypothesis on $\rho$ and $h$ (pluri-subharmonicity of $\rho$
and strict pluri-subharmonicity of $h$) imply $\left\langle \partial\overline{\partial}\rho;\varphi',\overline{\varphi'}\right\rangle (0)=\left\langle \partial\overline{\partial}h;\psi',\overline{\psi'}\right\rangle (0)=0$
and the last equality implies $\psi'(0)=0$. Moreover Lemma 8.1 of
\cite{CD08} implies that, for $1\leq j\leq k-1$, $\triangle^{j}\left(\left\langle \partial\overline{\partial}h;\psi',\overline{\psi'}\right\rangle \right)(0)=0$,
and, by induction, a simple calculation shows that this implies $\psi^{(j+1)}(0)=0$,
$1\leq j\leq k-1$. Then all the derivatives of order $\leq k$ of
$h\circ\psi$ vanish at the origin, and, thus, the same is true for
the derivatives of $\rho\circ\varphi$.

This implies first that the order of $r\circ\Phi$
cannot be infinite at $0$. Assume it is $l$, and let $k$ be a positive
integer such that $l=2k+1$ or $l=2k+2$. In both cases, the orders
of $\rho\circ\varphi-\rho\left(z^{1}\right)$, $h\circ\varphi-h\left(w^{1}\right)$
and $\psi-\psi(0)$ are $\geq k+1$ and we have $\frac{\mathrm{ord}_{0}\left(r\circ\Phi\right)}{\mbox{ord}_{0}(\Phi-\Phi(0))}\leq\frac{2\mathrm{ord}_{0}\left(\rho\circ\varphi-\rho\left(z^{1}\right)\right)}{\mathrm{ord}_{0}(\Phi-\Phi(0))}$
which implies $\tau_{\Phi}\leq2$ if $\mbox{ord}_{0}\left(\varphi-\varphi(0)\right)\geq k+1$
and $\tau_{\Phi}\leq2\mathrm{typ}_{z^{1}}(\rho)$ if not.

The case $p\geq2$ follows easily. If, for all $i$, $1\leq i\leq p$,
$w_{i}^{1}\neq0$ then $h$ is strictly pluri-subharmonic at $w^{1}$
and the previous proof applies. Otherwise, to simplify notations,
we can assume that $w_{i}^{1}\neq0$ for $1\leq i\leq r<p$ and $w_{i}^{1}=0$
for $r+1\leq i\leq p$. Denoting $u=\left(w_{1},\ldots w_{r}\right)$,
$v=\left(w_{r+1},\ldots,w_{p}\right)$, $\rho_{1}\left(z,u\right)=\rho(z)+\sum_{i=1}^{r}h_{i}\left(w_{i}\right)$
and $h_{1}(v)=\sum_{i=r+1}^{p}h_{i}\left(w_{i}\right)$, the previous
case shows that $\rho_{1}$ is pluri-subharmonic and the type of $\rho_1$
at $\left(z^{1},u^{1}\right)$ is bounded by $2\mathrm{typ}_{z^{1}}(\rho)$.
The conclusion is obtained applying \propref{Local-rho+h-negative-finite-type}.
\end{proof}

\subsection{\label{sec:Geometric-separation}Geometric separation}

If the domain $\Omega$ is completely geometrically separated at a
boundary point $z_{0}$ (see \cite{CD08} for definition), we do not
know, in general, if $\widetilde{\Omega}$ has the same property at
the point $\left(z_{0},0\right)$. We can only prove the weaker following
result (for which we will not give a proof because we do not have any
application):
\begin{spprop}
\label{prop:Geometric-Separation-Omega-tilde-General}Assume that
$\Omega$ is of finite type at $z_{0}\in\partial\Omega$ and that
$h(w)=\sum\left|w_{i}\right|^{2q_{i}}$, $w_{i}\in\mathbb{C}$. Then
for all Diederich-Forn\ae ss defining function $\rho$ of $\Omega$
of the form $\rho=\sigma e^{-L\left|z\right|^{2}}$ (see \cite{DF77-Strict-Psh-Exhau-Func-Inventiones})
with $L$ large enough (depending only on $\Omega$), we have:

if there exist a neighborhood $V$ of $z_{0}$, $K>0$ and a finite
dimensional vector space $E_{0}$ of $\left(1,0\right)$ vector fields
tangent to $\rho$ in $V$ (i.e. $L\rho\equiv0$ for $L\in E_{0}$)
such that, at any point of $V\cap\Omega$ and for
any $\delta>0$ there exists a $\left(K,\delta\right)$-extremal basis
for $\rho$ whose elements belong to $E_{0}$, then $\widetilde{\Omega}$
is geometrically separated at $\left(z_{0},0\right)$.
\end{spprop}
Note that the hypothesis in this proposition is stronger than the
simple fact that $\partial\Omega$ is geometrically separated at $z_{0}$:
the existence of extremal basis is assumed not only on the points
of $\partial\Omega\cap V$ but on all $\Omega\cap V$ (condition which
depends not only on $\Omega$ but also on the choice of $\rho$).
Unfortunately, if we add the hypothesis that all the level set of
$\rho$ are ``completely geometrically separated'' in $\Omega\cap V$
we can not prove, in general, that $\widetilde{\Omega}$
has the same property at $\left(z_{0},0\right)$. The only general
result we have is when $\Omega$ is in $\mathbb{C}^{2}$ (see \remref{geom-sep-convex-domains}):
\begin{stthm}
\label{thm:C-2-rho-plus-h-Loc-Diag}Assume $\Omega$ is pseudo-convex
of finite type in $\mathbb{C}^{2}$. Assume that $\rho=\rho_{s}$
with $\rho_{s}$ satisfying (\ref{eq:hypothesis-defining-func}) as
stated in \secref{Diederich-Fornaess-def-func} and that $h$ satisfies
\hyporef{Hypothesis-III} of \secref{Hypothesis_on_function_h}. Then
the domain 
\[
\widetilde{\Omega}=\left\{ \left(z,w\right)=\left(z,w_{1},\ldots,w_{m}\right)\in\mathbb{C}^{2}\times\mathbb{C}^{m}\mbox{ s. t. }r\left(z,w\right)=\rho(z)+\sum_{i=1}^{m}h_{i}\left(w_{i}\right)<0\right\} 
\]
 is pseudo-convex of finite type and has a Levi form which is locally
diagonalizable at every point of its boundary. In particular $\widetilde{\Omega}$
is completely geometrically separated (c.f. \cite{CD08}).\end{stthm}
\begin{proof}
Let $\left(z^{0},w^{0}\right)$ be a boundary point of $\widetilde{\Omega}$.
If $w_{i}^{0}\neq0$ for all $i$, then, by (\ref{2-Cor-global-rho-plus-h-finite-type})
of \corref{Global_finite-type_rho+h} following \propref{Local-rho+h-negative-finite-type},
$\partial\widetilde{\Omega}$
is strictly pseudo-convex at $\left(z^{0},w^{0}\right)$. Thus we
have only two cases to consider:
\begin{enumerate}
\item $w^{0}=0$;
\item \label{Case2-proof-Prop-rho-plus-h-Loc-Diag}there exist $i$ and
$j$ such that $w_{i}^{0}=0$ and $w_{j}^{0}\neq0$.
\end{enumerate}
Let us consider the first case. Denote by $L$ (resp. $N$) a non-vanishing vector
field ``complex tangent'' (resp. normal) to $\rho$ in a neighborhood
of $z^{0}$ (i.e. $L$ is of type $\left(1,0\right)$ and $L\rho\equiv0$).
We assume that $N$ is chosen so that $N\rho\equiv1$
in that neighborhood. Without changing the notation, we will consider
these vector fields defined in a neighborhood of $\left(z^{0},0\right)$
so that $Lr\equiv0$ and $Nr\equiv1$ in this neighborhood.
Let us define $m$ vector fields, $Z_{i}$, ``complex tangent'' to $r$,
in a neighborhood of $\left(z^{0},0\right)$ by
\[
Z_{i}=\frac{\partial}{\partial w_{i}}-\frac{\partial h_{i}}{\partial w_{i}}N,
\]
and then $m$ new vector fields, $W_{i}$, also ``complex tangent'' to
$r$, by
\begin{eqnarray*}
W_{1} & = & Z_{1},\\
W_{k+1} & = & Z_{k+1}-\sum_{j=1}^{k}a_{k+1}^{j}W_{j}\mbox{ for }k\geq2.
\end{eqnarray*}

We now show, by induction over $k$, that it is possible to choose
the coefficients $a_{k}^{j}$ so that the coefficient of the Levi
form of $r$, $\left[W_{k},\overline{W_{k'}}\right]\left(\partial r\right)$,
vanishes identically on the neighborhood of $\left(z^{0},0\right)$.
To simplify notations, in this proof, the character $\divideontimes$
will denote a $\mathcal{C}^{\infty}$ function in a neighborhood of
the origin.

Suppose that the vector fields $W_{i}$, $1\leq i\leq k$, have been
constructed with coefficients $a_{i}^{j}$, $2\leq i\leq k$, $1\leq j\leq i-1$
satisfying the two following properties:
\begin{enumerate}
\item \label{1-Claim-Prop_Loc-Diag}$a_{i}^{j}=\divideontimes\frac{\partial h_{i}}{\partial w_{i}}$,
\item \label{2-Claim-Prop-Loc-Diag}$\left[W_{i},\overline{W_{i}}\right]\left(\partial r\right)=\gamma_{i}\frac{\partial^{2}h_{i}}{\partial w_{i}\partial\overline{w_{i}}}$,
where $\gamma_{i}$ is a $\mathcal{C}^{\infty}$ real function in
a neighborhood of the origin of modulus greater than $\nicefrac{1}{2}$,
\end{enumerate}
\noindent and let us prove that $W_{k+1}$ can be constructed, so that the
coefficients $a_{k+1}^{j}$, $1\leq j\leq k$ satisfy the above
conditions. Note that the hypotheses made on $h_{i}$ imply first
that (\ref{2-Claim-Prop-Loc-Diag}) follows from (\ref{1-Claim-Prop_Loc-Diag})
because (\ref{1-Claim-Prop_Loc-Diag}) implies
\[
\left[W_{i},\overline{W_{i}}\right]\left(\partial r\right)=\frac{\partial^{2}h_{i}}{\partial w_{i}\partial\overline{w_{i}}}+\left|\frac{\partial h_{i}}{\partial w_{i}}\right|^{2}\left[N,\overline{N}\right]\left(\partial\rho\right)+\divideontimes\frac{\partial h_{i}}{\partial w_{i}}+\divideontimes\frac{\partial h_{i}}{\partial\overline{w_{i}}}.
\]
Note also that $W_{1}$ satisfy trivially (\ref{2-Claim-Prop-Loc-Diag}).

Thus we construct $W_{k+1}$ with coefficients $a_{k+1}^{j}$ satisfying
(\ref{1-Claim-Prop_Loc-Diag}).

For $j\leq k$, by induction, we have
\[
\left[W_{k+1},\overline{W_{j}}\right]\left(\partial r\right)=-\frac{\partial h_{j}}{\partial\overline{w_{j}}}\left[N,\overline{W_{j}}\right]\left(\partial r\right)-a_{k+1}^{j}\left[W_{j},\overline{W_{j}}\right]\left(\partial r\right)
\]
with
\begin{eqnarray*}
\left[N,\overline{W_{j}}\right]\left(\partial r\right) & = & -\frac{\partial h_{j}}{\partial\overline{w_{j}}}\left[N,\overline{N}\right]\left(\partial\rho\right)-\sum_{1\leq l<j}\overline{a_{j}^{l}}\left[N,\overline{W_{l}}\right]\left(\partial r\right)\\
 & = & \divideontimes\frac{\partial h_{j}}{\partial\overline{w_{j}}},
\end{eqnarray*}
if $2\leq j\leq k$, and
\[
\left[N,\overline{W_{1}}\right]\left(\partial r\right)=\divideontimes\frac{\partial h_{1}}{\partial\overline{w_{1}}}.
\]
This shows, by \hyporef{Hypothesis-III} of $h$, that the $a_{k+1}^{j}$
can be defined satisfying (\ref{1-Claim-Prop_Loc-Diag}) and such
that $\left[W_{k+1},\overline{W_{j}}\right]\left(\partial r\right)\equiv0$.

To finish the proof of the first case, we modify the vector field
$L$ replacing it by $L_{1}=L-\sum_{k=1}^{m}b_{k}W_{k}$ choosing
the $b_{k}$ so that the basis $\left(L_{1},W_{1},\ldots,W_{m}\right)$
diagonalizes the Levi form of $r$ in a neighborhood of $\left(z^{0},0\right)$
which means, now, $\left[L_{1},\overline{W_{i}}\right]\left(\partial r\right)\equiv0$
in that neighborhood:
\begin{eqnarray*}
\left[L_{1},\overline{W_{i}}\right]\left(\partial r\right) & = & \left[L,\overline{W_{i}}\right]\left(\partial r\right)-b_{i}\left[W_{i},\overline{W_{i}}\right]\left(\partial r\right)\\
 & = & \begin{cases}
-\frac{\partial h_{i}}{\partial\overline{w_{i}}}\left[L,\overline{N}\right]\left(\partial r\right)-\sum_{1\leq l<i}\overline{a_{i}^{l}}\left[L,\overline{W_{i}}\right]\left(\partial r\right)-b_{i}\left[W_{i},\overline{W_{i}}\right]\left(\partial r\right) & \mbox{if }i\geq2\\
-\frac{\partial h_{1}}{\partial\overline{w_{1}}}\left[L,\overline{N}\right]\left(\partial r\right)-b_{i}\left[W_{i},\overline{W_{i}}\right]\left(\partial r\right) & \mbox{if }i=1
\end{cases}\\
 & = & \divideontimes\frac{\partial h_{i}}{\partial\overline{w_{i}}}-b_{i}\left[W_{i},\overline{W_{i}}\right]\left(\partial r\right),
\end{eqnarray*}
and, by (\ref{2-Claim-Prop-Loc-Diag}), $b_{i}$ can be chosen $\mathcal{C}^{\infty}$
in a neighborhood of $\left(z^{0},0\right)$.

Let us now consider the second case (\ref{Case2-proof-Prop-rho-plus-h-Loc-Diag}).
To simplify the notations, we assume that $w_{i}^{0}\neq0$ for $1\leq i\leq m_{0}<m$
and $w_{i}^{0}=0$ for $m_{0}+1\leq i\leq m$. We denote $w=\left(w',w''\right)$,
with $w'=\left(w_{1},\ldots,w_{m_{0}}\right)$, $w''=\left(w_{m_{0}+1},\ldots,w_{m}\right)$,
$\rho^{1}\left(z,w'\right)=\rho(z)+\sum_{i=1}^{m_{0}}h_{i}\left(w_{i}\right)$
and $h^{1}\left(w''\right)=\sum_{i=m_{0}+1}^{m}h_{i}\left(w_{i}\right)$.

By \propref{Local-pseudoconvexity-for-rho-plus-h}, in a neighborhood
of $\left(z^{0},w'^{0}\right)$, the hypersurface $\left\{ \rho^{1}=0\right\} $
is strictly pseudo-convex. Then, reducing eventually the neighborhood,
there exists a basis of vector fields $\left(L_{1},\ldots,L_{m_{0}-1}\right)$
``complex tangent'' to $\rho^{1}$ (i.e. $L_{i}\left(\rho^{1}\right)\equiv0$)
which diagonalizes the Levi form of
$\rho^{1}$ in that neighborhood. Let us denote by $N$ the complex
normal vector field to $\rho^{1}$ in that neighborhood such that
$N\rho^{1}\equiv1$ (note that, reducing the neighborhood if necessary,
we can assume that the gradient of $\rho^{1}$ does not vanishes in
the neighborhood). We now consider the following $m-m_{0}$ vector
fields (which are ``complex tangent'' to $\rho^{1}$ in the neighborhood)
\begin{eqnarray*}
W_{1} & = & \frac{\partial}{\partial w_{m_{0}+1}}-\frac{\partial h_{m_{0}+1}}{\partial w_{m_{0}+1}}N-\sum_{i=1}^{m_{0}-1}a_{1}^{i}L_{i},\\
W_{j} & = & \frac{\partial}{\partial w_{m_{0}+j}}-\frac{\partial h_{m_{0}+j}}{\partial w_{m_{0}+j}}N-\sum_{i=1}^{m_{0}-1}a_{j}^{i}L_{i}-\sum_{l=1}^{j-1}b_{j}^{l}W_{l},\mbox{ for }j\geq2.
\end{eqnarray*}

To finish the proof of the theorem, we show that it is possible to
choose the coefficients $a_{j}^{i}$ and $b_{j}^{l}$ $\mathcal{C}^{\infty}$
in a neighborhood of $\left(z^{0},w{}^{0}\right)$ so that the basis
of vector field $\left(L_{1},\ldots,L_{m_{0}-1},W_{1},\ldots,W_{m-m_{0}}\right)$
diagonalizes the Levi form of $r$ in that neighborhood. We do this
using an induction argument similar to the one used in the first case:
assume that the vector fields $W_{j}$, $1\leq j\leq k$ have been
constructed and that their coefficients satisfy
\begin{enumerate}
\item \label{2-1-Claim-Prop_Loc-Diag}$a_{j}^{i}\mbox{ and }b_{j}^{l}=\divideontimes\frac{\partial h_{m_{0}+j}}{\partial w_{m_{0}+j}}$,
where $\divideontimes$ is a $\mathcal{C}^{\infty}$ function in a
neighborhood of $\left(z^{0},w{}^{0}\right)$,
\item \label{2-2-Claim-Prop-Loc-Diag-}$\left[W_{j},\overline{W_{j}}\right]\left(\partial r\right)=\gamma_j\frac{\partial^{2}h_{m_{0}+j}}{\partial w_{m_{0}+j}\partial\overline{w_{m_{0}+j}}}$,
where $\gamma_j$ is a $\mathcal{C}^{\infty}$ real function in a neighborhood
of $\left(z^{0},w{}^{0}\right)$ greater, in modulus, than $\nicefrac{1}{2}$.
\end{enumerate}
As for the first case, note that (\ref{2-2-Claim-Prop-Loc-Diag-})
follows (\ref{2-1-Claim-Prop_Loc-Diag}). Then, for $1\leq j\leq m_{0}-1$,
\[
\left[W_{k+1},\overline{L_{j}}\right]\left(\partial r\right)=-\frac{\partial h_{m_{0}+k+1}}{\partial w_{m_{0}+k+1}}\left[N,\overline{L_{j}}\right]\left(\partial r\right)-a_{k+1}^{j}\left[L_{j},\overline{L_{j}}\right]\left(\partial r\right),
\]
and, for $j\leq k$,
\[
\left[W_{k+1},\overline{W_{j}}\right]\left(\partial r\right)=-\frac{\partial h_{m_{0}+k+1}}{\partial w_{m_{0}+k+1}}\left[N,\overline{W_{j}}\right]\left(\partial r\right)-b_{k+1}^{j}\left[W_{j},\overline{W_{j}}\right]\left(\partial r\right),
\]
and the results follow, noting that $\left[L_{j},\overline{L_{j}}\right]\left(\partial r\right)$
is bounded from below by a strictly positive constant in a neighborhood
of $\left(z^{0},w{}^{0}\right)$, and, as in the first case, that
$\left[N,\overline{W_{j}}\right]\left(\partial r\right)=\divideontimes\frac{\partial h_{m_{0}+j}}{\partial\overline{w_{m_{0}+j}}}$,
with $\divideontimes$ is a $\mathcal{C}^{\infty}$ function in a
neighborhood of $\left(z^{0},w{}^{0}\right)$.\end{proof}
\begin{example*}
If $h\left(w\right)=\sum h_{i}\left(w_{i}\right)$, $w_{i}\in\mathbb{C}$,
where each function $h_{i}$ is a positive radial analytic function vanishing
at the origin, then the hypothesis of the theorem are verified.
\end{example*}
The proof of the second case, shows that \thmref{C-2-rho-plus-h-Loc-Diag}
is also valid if $\Omega$ is a smooth strictly pseudo-convex domain
in $\mathbb{C}^{n}$ (of any dimension). Moreover, applying first
the method of the second case and then the one of the first case,
this is also true when the rank of the Levi form of $\rho$ is $\geq n-2$.
Thus:
\begin{stthm}
Assume that $\rho=\rho_{s}$ with $\rho_{s}$ satisfying (\ref{eq:hypothesis-defining-func})
as stated in \secref{Diederich-Fornaess-def-func} and that $h$ satisfy
\hyporef{Hypothesis-III} of \secref{Hypothesis_on_function_h}. If
the rank of the Levi form of $\rho$ is $\geq n-2$, then $\widetilde{\Omega}$
is locally diagonalizable at every point of its boundary.\end{stthm}
\begin{srrem}
\label{rem:geom-sep-convex-domains}If $\Omega$ is a smooth bounded
convex domain of finite type in $\mathbb{C}^{n}$ we do not know if
it is always possible to choose a defining function $\rho$ and a
function $h$ so that $\widetilde{\Omega}$ is ``completely geometrically
separated'' at any boundary point (we will see in \secref{convex-case}
that this is possible near $\left\{ w=0\right\} $).
\end{srrem}

\section{\label{sec:Relations-between-operators-omega-tilde-omega}Relations
between operators related to \texorpdfstring{$\widetilde{\Omega}$}{Omega-tilde}
and to \texorpdfstring{$\Omega$}{Omega}}

Assume that $\Omega$ is a smooth bounded pseudo-convex domain of
finite type in $\mathbb{C}^{n}$, that $\rho=\rho_{s}$ is a defining
function of $\Omega$ where $\rho_{s}$ satisfies (\ref{eq:hypothesis-defining-func})
as stated in \secref{Diederich-Fornaess-def-func}, and that $h$
satisfies, at least, \hyporef{Hypothesis-II} of \secref{Hypothesis_on_function_h}.
Thus, by \corref{Global_finite-type_rho+h} of \propref{Local-rho+h-negative-finite-type},
the domain
\[
\widetilde{\Omega}=\left\{ \left(z,w\right)\in\mathbb{C}^{n+m}\mbox{ s. t. }\rho(z)+h(w)<0\right\} ,
\]
is a smooth bounded pseudoconvex domain of $\mathbb{C}^{n+m}$ of
finite type.

In this section we assume that the functions $h_{i}$ defining $h$
are radial (i.e. $h(w)=\sum h_{i}\left(\left|w_{i}\right|\right)$,
and, taking into account the properties of the $\overline{\partial}$-Neumann
problem for $\widetilde{\Omega}$, we derive properties of solutions
of the $\overline{\partial}$-equation and  properties of the Bergman projections
related to the weight
\begin{equation}
\omega(z)=\int_{\left\{ h(w)<-\rho(z)\right\} }d\lambda(w).\label{eq:weight-on-Omega}
\end{equation}

Suppose $f=\sum_{i=1}^{n}f_{i}d\overline{z_{i}}$ is a $\left(0,1\right)$-form
on $\Omega$. Consider it as a $\left(0,1\right)$-form $\widetilde{f}$
in $\widetilde{\Omega}$. If $f$ is $\overline{\partial}$-closed, then
so is $\widetilde{f}$, and if $\widetilde{u}$ is a solution of $\overline{\partial}\widetilde{u}=\widetilde{f}$
in $\widetilde{\Omega}$, then $\widetilde{u}$ is holomorphic in
the variable $w$ and the function $u$ defined by $u(z)=\widetilde{u}(z,0)$
is a solution of the equation $\overline{\partial}u=f$ in $\Omega$.
Moreover, for all $\alpha\in\mathbb{N}^{n}$, denoting $D_{z}^{\alpha}=\frac{\partial^{\left|\alpha\right|}}{\partial z_{1}^{\alpha_{1}}\ldots,\partial z_{n}^{\alpha_{n}}}$,
we have $D_{z}^{\alpha}u(z)=D_{z}^{\alpha}\widetilde{u}(z,0)$, $w\mapsto D_{z}^{\alpha}\widetilde{u}(z,w)$
is holomorphic, for any $p\in\left[1,+\infty\right]$,
$w\mapsto\left|D_{z}^{\alpha}\widetilde{u}(z,w)\right|^{p}$
is pluri-subharmonic, and, by the mean value property (the functions $h_{i}$ being radial),
\[
D_{z}^{\alpha}u(z)=\left(\omega(z)\right)^{-1}\int_{\left\{ h(w)<-\rho(z)\right\} }D_{z}^{\alpha}\widetilde{u}(z,w)d\lambda(w),
\]
and
\[
\int_{\Omega}\left|D_{z}^{\alpha}u(z)\right|^{p}\omega(z)d\lambda(z)\leq\int_{\widetilde{\Omega}}\left|D_{z}^{\alpha}\widetilde{u}(z,w)\right|^{p}d\lambda(z,w).
\]

Thus:
\begin{sllem}
\label{lem:Relation-sol-d-bar-omega-omage-tilde}With the conditions
and notations stated above, for any $p\in\left[1,+\infty\right]$
and any integer $t\geq0$, denote by $L_{\omega}^{p,t}(\Omega)$ the
Sobolev space of functions $g$ (resp. of $\left(0,1\right)$-forms
$g=\sum_{i=1}^{n}g_{i}d\overline{z_{i}}$) such that, for all $\alpha\in\mathbb{N}^{n}$,
$\left|\alpha\right|\leq t$, $D_{z}^{\alpha}g$ belongs to the weighted
$L^{p}$ space $L_{\omega}^{p}(\Omega)=L^{p}(\Omega,\omega(z)d\lambda(z))$
(resp. to the weighted space $L_{\left(0,1\right),\omega}^{p,t}(\Omega)$
of $\left(0,1\right)$-forms $g$ on $\Omega$ whose coefficients
$g_{i}$ belong to $L_{\omega}^{p,t}(\Omega)$) equipped with the
norm $\left\Vert g\right\Vert _{L_{\omega}^{p,t}(\Omega)}=\sum_{\left|\alpha\right|\leq t}\left\Vert D_{z}^{\alpha}g\right\Vert _{L_{\omega}^{p}(\Omega)}$
(resp. $\left\Vert g\right\Vert _{L_{\left(0,1\right),\omega}^{p,t}(\Omega)}=\sum_{i=1}^{n}\sum_{\left|\alpha\right|\leq t}\left\Vert D_{z}^{\alpha}g_{i}\right\Vert _{L_{\omega}^{p}(\Omega)}$).
Then
\begin{enumerate}
\item $f$ is in $L_{\left(0,1\right),\omega}^{p,t}(\Omega)$ if an only
if $\widetilde{f}$ belongs to $L_{\left(0,1\right)}^{p,t}(\widetilde{\Omega})$
and, in this case, $\left\Vert f\right\Vert _{L_{\left(0,1\right),\omega}^{p,t}(\Omega)}=\left\Vert \widetilde{f}\right\Vert _{L_{\left(0,1\right)}^{p,t}\left(\widetilde{\Omega}\right)}$;
\item If $\widetilde{u}$ belongs to $L^{p,t}\left(\widetilde{\Omega}\right)$
then $u$ belongs to $L_{\omega}^{p,t}(\Omega)$ and $\left\Vert u\right\Vert _{L_{\omega}^{p,t}(\Omega)}\leq\left\Vert \widetilde{u}\right\Vert _{L^{p,t}\left(\widetilde{\Omega}\right)}$.
\end{enumerate}
\end{sllem}

Similarly, a function $u$ belongs to $L_{\omega}^{2}(\Omega)$ if
and only if the function $\widetilde{u}$, defined on $\widetilde{\Omega}$
by $\widetilde{u}(z,w)=u(z)$, belongs to $L^{2}\left(\widetilde{\Omega}\right)$.
So, if $P^{\widetilde{\Omega}}$ denotes the Bergman projection of
$\widetilde{\Omega}$ and $P_{\omega}^{\Omega}$ the Bergman projection
of $\Omega$ with the weight $\omega$, as $h(w)=\sum_{i=1}^{p}h_{i}\left(\left|w_{i}\right|\right)$,
$w_{i}\in\mathbb{C}^{m_{i}}$, by the mean value property we have
$P_{\omega}^{\Omega}(u)(z)=P^{\widetilde{\Omega}}\left(\widetilde{u}\right)(z,0)$.
Then:
\begin{sllem}
\label{lem:Relation-Bergman-omega-omega-tilde}With the above notations,
\begin{enumerate}
\item We have $\left\Vert u\right\Vert _{L_{\omega}^{p,t}(\Omega)}=\left\Vert \widetilde{u}\right\Vert _{L^{p,t}\left(\widetilde{\Omega}\right)}$;
\item We have $P_{\omega}^{\Omega}(u)(z)=P^{\widetilde{\Omega}}\left(\widetilde{u}\right)(z,0)$
and $\left\Vert P_{\omega}^{\Omega}(u)\right\Vert _{L_{\omega}^{p,t}(\Omega)}\leq\left\Vert P^{\widetilde{\Omega}}\left(\widetilde{u}\right)\right\Vert _{L^{p,t}\left(\widetilde{\Omega}\right)}$.
\item For any $\overline{\partial}$-closed $\left(0,1\right)$-form $f\in L_{\omega}^{2}(\Omega)$,
denoting $\widetilde{f}(z,w)=f(z)$, $\overline{\partial}^{*}\mathcal{N}_{\widetilde{\Omega}}\left(\widetilde{f}\right)(z,0)$
is the solution of the equation $\overline{\partial}u=f$ orthogonal
to holomorphic functions in $L_{\omega}^{2}(\Omega)$.
\item If $K_{B,\omega}^{\Omega}$ (resp. $K_{B}^{\widetilde{\Omega}}$)
denotes the Bergman kernel of $\Omega$ associated to the measure
$\omega(z)d\lambda(z)$ (resp. of $\widetilde{\Omega}$ associated
to the Lebesgue measure), we have $K_{B,\omega}^{\Omega}(z,\zeta)=K_{B}^{\widetilde{\Omega}}\left((z,0),(\zeta,0)\right)$.
\end{enumerate}
\end{sllem}

Now, we will derive from these lemmas some simple weighted estimates
on $\Omega$ when the corresponding unweighted estimates are known
on $\widetilde{\Omega}$.

\subsection{Sobolev estimates for general pseudo-convex domain}

As $\widetilde{\Omega}$ is of finite type, by the fundamental result
of D. Catlin (\cite{Cat87}) the $\overline{\partial}$-Neumann problem
of $\widetilde{\Omega}$ satisfies a subelliptic estimate. Then, all
the associated operators map continuously the $L^{2}$ Sobolev spaces
of $\widetilde{\Omega}$ into themselves.

To respect traditional notations, denote, for $t\in\mathbb{N}$,
$W_{\omega}^{t}(\Omega)=L_{\omega}^{2,t}(\Omega)$ and $W_{\left(0,1\right),\omega}^{t}(\Omega)=L_{\left(0,1\right),\omega}^{2,t}(\Omega)$.

Lemmas \ref{lem:Relation-Bergman-omega-omega-tilde} and \ref{lem:Relation-sol-d-bar-omega-omage-tilde}
imply thus:
\begin{stthm}
Let $\Omega$ be a smooth bounded pseudo-convex domain of finite type
in $\mathbb{C}^{n}$. Let $\rho=\rho_{s}$ be a defining function
of $\Omega$ where $\rho_{s}$ satisfies (\ref{eq:hypothesis-defining-func})
as stated in \secref{Diederich-Fornaess-def-func}. Let $h$ be a
smooth function on $\mathbb{C}^{m}$ satisfying \hyporef{Hypothesis-II}
of \secref{Hypothesis_on_function_h} the functions $h_{i}$ being
radial. Then, $\omega$ being the weight defined by (\ref{eq:weight-on-Omega}):
\begin{enumerate}
\item For any integer $t$, if $f$ is a $\overline{\partial}$-closed $\left(0,1\right)$-form
in $W_{\left(0,1\right),\omega}^{t}(\Omega)$, then the solution to
the equation $\overline{\partial}u=f$ orthogonal to holomorphic functions
in $L_{\omega}^{2}(\Omega)$ satisfies $\left\Vert u\right\Vert _{W_{\omega}^{t}(\Omega)}\leq C\left\Vert f\right\Vert _{W_{\left(0,1\right),\omega}^{t}(\Omega)}$,
the constant $C$ depending on $\rho$, $h$ and $t$;
\item For any integer $t$, the weighted Bergman projection $P_{\omega}^{\Omega}$
maps the Sobolev space $W_{\omega}^{t}(\Omega)$ continuously into
itself.
\end{enumerate}
\end{stthm}
\thmref{L2-Sobolev-Est-General} is (2) of the following corollary:
\begin{cor*}
Let $\Omega$ and $\rho$ be as in the theorem. Let $r\geq0$ be a
rational number and $\varphi_{r}$ be the pluri-subharmonic function
$\varphi_{r}=-r\log(-\rho)$ (c.f. \remref{-log-of--D-F-def-function-psh}).
Let us denote by $\mathcal{N}_{\varphi_{r}}$ the $\overline{\partial}$-Neumann
operator for the weight $e^{-\varphi_{r}}$ acting on $\left(0,1\right)$-forms
and by $\overline{\partial}_{\varphi_{r}}^{*}\mathcal{N}_{\varphi_{r}}^{\left(0,1\right)}$
the restriction to the space of $\overline{\partial}$-closed forms
in $L_{\left(0,1\right),\left(-\rho\right)^{r}}^{2}(\Omega)$ of the
operator $\overline{\partial}_{\varphi_{r}}^{*}\mathcal{N}_{\varphi_{r}}$
giving the $L_{\left(-\rho\right)^{r}}^{2}$ minimal solution of the
$\overline{\partial}$-equation. Let us denote by $\mathcal{B}_{\varphi_{r}}$
the Bergman projection of $L_{\varphi_{r}}^{2}(\Omega)$. Then:
\begin{enumerate}
\item For all $t\geq0$, $\overline{\partial}_{\varphi_{r}}^{*}\mathcal{N}_{\varphi_{r}}^{\left(0,1\right)}$
maps the subspace of $\overline{\partial}$-closed forms
of $W_{\left(0,1\right),\left(-\rho\right)^{r}}^{t}(\Omega)$ continuously into
$W_{\left(-\rho\right)^{r}}^{t}(\Omega)$.
\item For all real number $t$, $\mathcal{B}_{\varphi_{r}}$ maps
$W_{\left(-\rho\right)^{r}}^{t}(\Omega)$ continuously into itself.
\end{enumerate}
\end{cor*}
\begin{srrem}
\label{rem:Remark-Sobolev-L2}\quad\mynobreakpar
\begin{enumerate}
\item \label{1-Remark-Sobolev-L2}In the corollary, the function $h$ is equal to $\sum\left|w_{i}\right|^{2q_{i}}$,
$w_{i}\in\mathbb{C}$, the integers $q_{i}$ being chosen so that
$r=\sum\frac{1}{q_{i}}$. $\widetilde{\Omega}$ being of finite type,
there is a gain in the Sobolev scale for the estimates of the $\overline{\partial}$-Neumann
problem on $\widetilde{\Omega}$. This implies a similar gain for
$\overline{\partial}_{\varphi_{r}}^{*}\mathcal{N}_{\varphi_{r}}^{\left(0,1\right)}$.
But this gain is the inverse on the type of $\widetilde{\Omega}$
which is given in Corollary 1 of \propref{Local-rho+h-negative-finite-type}
and, then, can be very small depending on $r$.
\item If $\Omega$ is a smooth bounded pseudo-convex domain in $\mathbb{C}^{n}$
(\emph{not assumed of finite type}) admitting a defining function $\rho$ which
is pluri-subharmonic in $\overline{\Omega}$ then:

\begin{enumerate}
\item If $h$ is a positive pluri-subharmonic function satisfying $\nabla h(w)\neq0$
if $w\neq0$, $\lim_{\left|w\right|\rightarrow+\infty}h(w)=+\infty$
and $h(w)=\sum_{i=1}^{p}h_{i}\left(\left|w_{i}\right|\right)$, $w=\left(w_{1},\ldots,w_{p}\right)$,
$w_{i}\in\mathbb{C}^{m_{i}}$, then $\widetilde{\Omega}$ is bounded, admits a
pluri-subharmonic defining function and, applying a theorem of H.
Boas \& E. Straube (\cite{Boas-Straube-def-psh91}), we get that,
for all real number $t$, the weighted Bergman projection $P_{\omega}^{\Omega}$
maps the Sobolev spaces $W_{\omega}^{t}(\Omega)$ continuously into
themselves.
\item Moreover, if $\rho$ is of finite type in $\overline{\Omega}$ then
\corref{Global-finite-type-psh-def-func} of \propref{Local-rho+h-negative-finite-type}
shows that the results of the theorem and the corollary are also valid
using $\rho$ to define $\widetilde{\Omega}$ and thus for other weights
$\omega$.
\end{enumerate}
\end{enumerate}
\end{srrem}

\subsection{\label{sec:Lipschitz-and-first-Lp-estimates-C2}Lipschitz estimates
for domains in $\mathbb{C}^{2}$}

Here we obtain estimates on weighted Bergman projections of $\Omega$
using only properties of the Bergman projection of the domain $\widetilde{\Omega}$.
For this we need $\widetilde{\Omega}$ to be ``completely geometrically
separated'' and we assume:

\medskip{}

\emph{$\Omega$ is a domain in $\mathbb{C}^{2}$, or the rank of the
Levi form of $\partial\Omega$ is $\geq n-2$, $\rho=\rho_{s}$ with
$\rho_{s}$ satisfying (\ref{eq:hypothesis-defining-func}) as stated
in \secref{Diederich-Fornaess-def-func} and $h$ satisfies \hyporef{Hypothesis-III}
of \secref{Hypothesis_on_function_h}. We denote by $\omega=\omega_{\rho,h}$
the associated weight (\eqref{weight-on-Omega}).}

\medskip{}

Let $M$ be the type of $\Omega$. By \cite{Charpentier-Dupain-Geometery-Finite-Type-Loc-Diag,Charpentier-Dupain-Szego-Barcelone},
we know that $P^{\widetilde{\Omega}}$ maps continuously the Lipschitz
space $\Lambda_{\alpha}(\widetilde{\Omega})$, $\alpha\geq0$, into
itself and that the space of holomorphic functions in $\Lambda_{\alpha}(\Omega)$
is continuously embedded in the anisotropic Lipschitz space $\Gamma_{\alpha}(\Omega)$
for $\alpha<\nicefrac{1}{M}$. Then \lemref{Relation-Bergman-omega-omega-tilde} gives immediately:
\begin{stthm}
In the conditions stated above for $\Omega$, $\rho$, $h$ and $\omega$,
the weighted Bergman projection $P_{\omega}^{\Omega}$ maps
the Lipschitz space $\Lambda_{\alpha}(\Omega)$ continuously into
itself for all $\alpha\geq0$ and into the anisotropic Lipschitz space
$\Gamma_{\alpha}(\Omega)$ for $\alpha<\nicefrac{1}{M}$.\end{stthm}
\begin{rem*}
\quad\mynobreakpar
\begin{enumerate}

\item In the next section, using pointwise estimates of the kernel of $P_{\omega}^{\Omega}$
we will extend the Lipschitz estimate for $P_{\omega}^{\Omega}$ to
convex domains of finite type in $\mathbb{C}^{n}$ but for a smaller
class of weights $\omega$.

\item In the conditions of the preceding theorem, choosing
$h(w)=\sum_{i}\left|w_{i}\right|^{q_{i}}$, $w_{i}\in\mathbb{C}$,
the weight $\omega$ is equal to (a constant times) $\left(-\rho\right)^{q}$, with
 $q=\sum\frac{1}{q_{i}}$. Using that the Bergman projection $P^{\widetilde{\Omega}}$ of
 $\widetilde{\Omega}$ maps the Sobolev spaces $L_{s}^{p}$ ($1<p<+\infty$,
$s\in\mathbb{N}$) continuously into themselves
 (\cite{Charpentier-Dupain-Geometery-Finite-Type-Loc-Diag,Charpentier-Dupain-Szego-Barcelone}),
we get immediately that the weighted Bergman projections $P_{\omega}^{\Omega}$
maps the (weighted) Sobolev spaces
$L_{s}^{p}\left(\left(-\rho\right)^{q}d\lambda\right)$
continuously into themselves and it is easy to extended this to the spaces
$L_{s}^{p}\left(\left(-\rho\right)^{t}d\lambda\right)$
for $-1<t-q<p-1$.\\
In the next section, establishing pointwise estimates of the kernel
of $P_{\omega}^{\Omega}$ we will get the (better) results stated
in \thmref{Sobolev-Lp_Lip_Bergman}.
\end{enumerate}
\end{rem*}

\section{\label{sec:Sharp-estimates-of-weighted-Bergman-kernel}Sharp estimates
of the weighted Bergman kernel and proof of \texorpdfstring{\thmref{Sobolev-Lp_Lip_Bergman}}{Theorem 1.1}}

The aim of this section is to establish precise pointwise estimates
of the kernel of the weighted Bergman projection $P_{\omega}^{\Omega}$
in terms of the geometry of $\Omega$ (from which we will deduce \thmref{Sobolev-Lp_Lip_Bergman})
using pointwise estimates of the kernel of $P^{\widetilde{\Omega}}$
and \lemref{Relation-Bergman-omega-omega-tilde}. Hence we need, at
least, that the domain $\widetilde{\Omega}$ is ``completely geometrically
separated'' near the set $\left\{ \left(z,0\right)\in\widetilde{\Omega}\right\} $
and to have a precise comparison of the geometries of $\widetilde{\Omega}$
and $\Omega$. We are able to do this in the two following cases:
\begin{itemize}
\item $\Omega$ is a finite type domain in $\mathbb{C}^{2}$, $\rho=\rho_{s}$
satisfies (\ref{eq:hypothesis-defining-func}) and $h$ satisfies
\hyporef{Hypothesis-IV} of \secref{Hypothesis_on_function_h}
\item $\Omega$ is a convex domain of finite type in $\mathbb{C}^{n}$,
$\rho=g^{4}e^{1-\nicefrac{1}{g}}-1$ where $g$ is a gauge function
of $\Omega$, so that $\rho$ is convex and of finite type in a neighborhood
of $\partial\Omega$, and $h$ satisfies \hyporef{Hypothesis-V} of
\secref{Hypothesis_on_function_h}
\end{itemize}
(the weight $\omega$ being given by (\ref{eq:weight-on-Omega})).
Recall that for a convex domain containing the origin of $\mathbb{C}^{n}$,
the gauge relative to $0$ is defined by
\[
g(z)=\inf\left\{ t\geq0\mbox{ such that }z\in t\Omega\right\} .
\]
Then it is easy to see that $g$ is convex, smooth outside the origin
and of finite type in a neighborhood of $\partial\Omega$. Thus the
defining function we choose is smooth.

Note that, for the convex case, we are not able, in general, to get
complete geometric separation of $\widetilde{\Omega}$ near $\partial\Omega\times\left\{ 0\right\} $
if we use a Diederich-Forn\ae ss defining function (see \propref{Geometric-Separation-Omega-tilde-General}).
This property being indispensable in this section we need to use another
defining function defined using the gauge: doing this, we loose the
property of finite type everywhere on $\partial\widetilde{\Omega}$
and so the global properties of the Bergman projection $P^{\widetilde{\Omega}}$
but local estimates of the kernel of $P^{\widetilde{\Omega}}$ will
suffice for our purpose.

\subsection{\label{sec:C2-case}The case of finite type domains in $\mathbb{C}^{2}$}

\subsubsection{\label{sec:Precise-comparison-geometries-C2}Precise comparison between
the geometries of \texorpdfstring{$\Omega$}{Omega} and \texorpdfstring{$\widetilde{\Omega}$}{Omega-tilde}}

We assume that the defining function $\rho$ of $\Omega$ is $\rho=\rho_{s}$
where $\rho_{s}$ satisfies (\ref{eq:hypothesis-defining-func}) as
stated in \secref{Diederich-Fornaess-def-func} and that the function
$h$ satisfies \hyporef{Hypothesis-IV} of \secref{Hypothesis_on_function_h}
(so that $\widetilde{\Omega}$ is of finite type and has a Levi form
locally diagonalizable at every point of $\partial\widetilde{\Omega}$).

We use the notations of the proof of \thmref{C-2-rho-plus-h-Loc-Diag}
for $L$, $N$, $L_{1}$ and $W_{k}$, and let us denote by $\widetilde{N}$
the complex normal to the defining function $r$ of
$\widetilde{\Omega}$ (i.e. $\widetilde{N}\left(r\right)\equiv1$
in a neighborhood of the boundary of $\widetilde{\Omega}$). Moreover,
for the geometries, using the notation ``$F^{\Omega}\left(L,z,\delta)\right)$''
introduced in Section 2 of \cite{CD08}, we denote $\widetilde{F}_{1}\left(\widetilde{\zeta},\delta\right)=\delta^{-2}$,
$\widetilde{F}_{2}\left(\widetilde{\zeta},\delta\right)=F^{\widetilde{\Omega}}\left(L_{1},\widetilde{\zeta},\delta\right)$,
$\widetilde{F}_{i}\left(\widetilde{\zeta},\delta\right)=F^{\widetilde{\Omega}}\left(W_{i-2},\widetilde{\zeta},\delta\right)$
and $F_{L}\left(\zeta,\delta\right)=F^{\Omega}\left(L,\zeta,\delta\right)$.

Let us first compare the weights $\widetilde{F}_{i}\left(\widetilde{\zeta},\delta\right)$
and $F_{L}(\zeta,\delta)$ constructed with the extremal basis defined
in \secref{Geometric-separation}.
\begin{sllem}
\label{lem:Lemma-1_kernel_n=00003D2}We have:
\begin{enumerate}
\item $L_{1}=L-\sum\divideontimes W_{k}=L-\sum\divideontimes\frac{\partial}{\partial w_{i}}-\left(\sum\divideontimes\frac{\partial h_{i}}{\partial w_{i}}\right)N$,
where $\divideontimes$ are $\mathcal{C}^{\infty}$ functions;
\item $N=\beta\widetilde{N}+\sum\left(\divideontimes\frac{\partial h_{i}}{\partial\overline{w_{i}}}+\sum_{j>i}\divideontimes\left|\frac{\partial h_{j}}{\partial w_{j}}\right|^{2}\right)W_{i}$,
where $\beta$ and $\divideontimes$ are $\mathcal{C}^{\infty}$ functions,
$\beta\simeq1$ for $\left|w\right|$ small.
\end{enumerate}
\end{sllem}
\begin{proof}
Part (1) is a trivial consequence of the definitions of the vector
fields. Let us give some indications for part (2). We have
\[
\widetilde{N}=\frac{\overline{\nabla\rho}+\overline{\nabla h}}{\left|\nabla\rho\right|^{2}+\left|\nabla h\right|^{2}}\mbox{ and }N=\frac{\overline{\nabla\rho}}{\left|\nabla\rho\right|^{2}}.
\]
Thus
\[
N=\beta\widetilde{N}+\divideontimes\sum\frac{\partial h_{i}}{\partial\overline{w_{i}}}\left(Z_{i}+\frac{\partial h_{i}}{\partial\overline{w_{i}}}N\right),
\]
and
\[
\left(1-\sum\divideontimes\left|\frac{\partial h_{i}}{\partial w_{i}}\right|^{2}\right)N=\beta\widetilde{N}+\sum\left(\divideontimes\frac{\partial h_{i}}{\partial\overline{w_{i}}}+\sum_{j>i}\divideontimes\left|\frac{\partial h_{j}}{\partial w_{j}}\right|^{2}\right)W_{i}.
\]

\end{proof}

Now we apply this lemma to estimate the weights $\widetilde{F}_{i}\left(\widetilde{\zeta},\delta\right)$
for $\widetilde{\zeta}=\left(\zeta,0\right)\in\partial\widetilde{\Omega}$.

Denoting $c_{11}=\left[L,\overline{L}\right]\left(\partial\rho\right)$,
we have
\begin{eqnarray*}
\widetilde{c}_{11} & = & \left[L_{1},\overline{L_{1}}\right]\left(\partial r\right)\\
 & = & \left[L-\sum\divideontimes\frac{\partial}{\partial w_{i}}-\left(\sum\divideontimes\frac{\partial h_{i}}{\partial\overline{w_{i}}}\right)N,\overline{L}-\sum\divideontimes\frac{\partial}{\partial\overline{w_{i}}}-\left(\sum\divideontimes\frac{\partial h_{i}}{\partial\overline{w_{i}}}\right)\overline{N}\right]\left(\partial r\right)\\
 & = & c_{11}+\sum\divideontimes\frac{\partial h_{i}}{\partial w_{i}}+\sum\divideontimes\frac{\partial h_{i}}{\partial\overline{w_{i}}}.
\end{eqnarray*}
Then the order of the functions $h_{i}$ being greater
than the type of $\Omega$, it is obvious that $\widetilde{F}_{1}\left(\widetilde{\zeta},\delta\right)\simeq F_{1}\left(\zeta,\delta\right)$.

Furthermore, in the proof of \thmref{C-2-rho-plus-h-Loc-Diag}, we saw that
$\left[W_{k},\overline{W_{k}}\right]\left(\partial r\right)=\alpha_{k}\frac{\partial^{2}h_{k}}{\partial w_{k}\partial\overline{w_{k}}}$,
$\alpha_{k}\in\mathcal{C}^{\infty}$, $\left|\alpha_{k}\right|\geq\nicefrac{1}{2}$,
then, as $W_{k}=\frac{\partial}{\partial w_{k}}+\divideontimes N+\sum_{i<k}\divideontimes\frac{\partial}{\partial w_{i}}$
and
\[
\widetilde{F}_{k+2}=\sum_{\mathcal{L}\in\mathcal{L}_{k}}\left(\frac{\mathcal{L}\left(\left[W_{k},\overline{W_{k}}\right]\left(\partial r\right)\right)}{\delta}\right)^{\nicefrac{2}{\left|\mathcal{L}\right|+2}},
\]
it follows clearly that
\[
\widetilde{F}_{k+2}\left(\widetilde{\zeta},\delta\right)\simeq\left(\frac{1}{\delta}\right)^{\nicefrac{1}{q_{k}}}\lesssim\widetilde{F}_{1}\left(\widetilde{\zeta},\delta\right).
\]

\medskip{}

Finally, we compare the pseudo-distances in $\Omega$ and $\widetilde{\Omega}$.

Let $\gamma$ and $\widetilde{\gamma}$ be the respective pseudo-distances
in $\partial\Omega$ and $\partial\widetilde{\Omega}$ defined by
the exponential map of tangent vectors fields associated to extremal
basis (see \cite{CD08} and \cite[p. 75 and 100]{Charpentier-Dupain-Geometery-Finite-Type-Loc-Diag}).
Then
\begin{sllem}
\label{lem:Lemma-2_kernel_n=00003D2}With the above notations, $\gamma\left(z,\zeta\right)\simeq\widetilde{\gamma}\left(\widetilde{z},\widetilde{\zeta}\right)$,
for $z$ and $\zeta$ in $\partial\Omega$.\end{sllem}
\begin{proof}
We use the notations introduced in \cite{Charpentier-Dupain-Geometery-Finite-Type-Loc-Diag}:
for $\Omega$, $\mathcal{Y}_{1}=\Re\mbox{e}N$, $\mathcal{Y}_{2}=\Im\mbox{m}N$,
$\mathcal{Y}_{3}=\Re\mbox{e}L_{1}$, $\mathcal{Y}_{4}=\Im\mbox{m}L_{1}$
and, for $\widetilde{\Omega}$, $\widetilde{\mathcal{Y}}_{1}=\Re\mbox{e}\widetilde{N}$,
$\widetilde{\mathcal{Y}}_{2}=\Im\mbox{m}\widetilde{N}$, $\widetilde{\mathcal{Y}}_{3}=\Re\mbox{e}\widetilde{L}_{1}$,$\widetilde{\mathcal{Y}}_{4}=\Im\mbox{m}\widetilde{L}_{1}$
and $\widetilde{\mathcal{Y}}_{2k+3}=\Re\mbox{e}W_{k}$, $\widetilde{\mathcal{Y}}_{2k+4}=\Im\mbox{m}W_{k}$.
Then, following \cite[p. 100]{Charpentier-Dupain-Geometery-Finite-Type-Loc-Diag},
there exists $\varphi\,:\left[0,1\right]\rightarrow\mathbb{C}^{2}$,
piecewise $\mathcal{C}^{1}$, such that $\varphi(0)=z$, $\varphi(1)=\zeta$,
$\varphi'(t)=\sum a_{i}(t)\mathcal{Y}_{i}(\varphi(t))$ a.e. with
$\left|a_{i}(t)\right|\lesssim G_{i}\left(z,\gamma\left(z,\zeta\right)\right)^{-\nicefrac{1}{2}}$,
where $G_{1}=G_{2}=\nicefrac{1}{\delta^{2}}$ and $G_{3}=G_{4}=F_{1}(z,\delta)$.
Now, let $\widetilde{\varphi}(t)=\left(\varphi(t),0\right)$ be the
same curve considered in $\mathbb{C}^{2+m}$. Then \lemref{Lemma-1_kernel_n=00003D2}
and $G_{1}(z,\delta)=\widetilde{G}_{1}\left(\widetilde{z},\delta\right)=\frac{1}{\delta^{2}}\gtrsim F_{1}(z,\delta)\simeq\widetilde{F}_{1}\left(\widetilde{z},\delta\right)\gtrsim\widetilde{F}_{k}\left(\widetilde{z},\delta\right)=\widetilde{G}_{2k+1}\left(\widetilde{z},\delta\right)=\widetilde{G}_{2k+2}\left(\widetilde{z},\delta\right),$
$k\geq2$) show that $\widetilde{\varphi}$ satisfies $\widetilde{\varphi}'(t)=\sum b_{i}(t)\widetilde{\mathcal{Y}}_{i}\left(\widetilde{\varphi}(t)\right)$
with $\left|b_{i}(t)\right|\lesssim\widetilde{G}_{i}\left(\widetilde{z},\gamma\left(z,\zeta\right)\right)^{-\nicefrac{1}{2}}$.
Thus
\[
\widetilde{\gamma}\left(\widetilde{z},\widetilde{\zeta}\right)\lesssim\gamma\left(z,\zeta\right).
\]

To show the converse inequality, consider a curve $\widetilde{\varphi}\,:\left[0,1\right]\rightarrow\mathbb{C}^{2+m}$
such that $\widetilde{\varphi}(0)=\widetilde{z}$, $\widetilde{\varphi}(1)=\widetilde{\zeta}$,
$\widetilde{\varphi}'(t)=\sum b_{i}(t)\widetilde{\mathcal{Y}}_{i}\left(\widetilde{\varphi}(t)\right)$
a.e. with $\left|b_{i}(t)\right|\lesssim\widetilde{G}_{i}\left(\widetilde{z},\widetilde{\gamma}\left(\widetilde{z},\widetilde{\zeta}\right)\right)^{-\nicefrac{1}{2}}$. 

First, we consider the component $\widetilde{\varphi}_{2+i}(t)$ of
the curve $\widetilde{\varphi}$. Let us decompose $\widetilde{\varphi}'$
on the basis $L$, $N$, $\frac{\partial}{\partial w_{i}}$ and their
conjugates. Then the coefficient $\widetilde{w}_{i}\left(\widetilde{\varphi}'(t)\right)$
of $\widetilde{\varphi}'(t)$ in the directions $\frac{\partial}{\partial w_{i}}$
or $\frac{\partial}{\partial\overline{w_{i}}}$ is, in modulus, bounded
from above by
\[
F_{1}\left(z,\widetilde{\gamma}\left(\widetilde{z},\widetilde{\zeta}\right)\right)^{-\nicefrac{1}{2}}+\widetilde{F}_{i}\left(\widetilde{z},\widetilde{\gamma}\left(\widetilde{z},\widetilde{\zeta}\right)\right)^{-\nicefrac{1}{2}}+\widetilde{\gamma}\left(\widetilde{z},\widetilde{\zeta}\right)\lesssim\widetilde{F}_{i}\left(\widetilde{z},\widetilde{\gamma}\left(\widetilde{z},\widetilde{\zeta}\right)\right)^{-\nicefrac{1}{2}}\lesssim\widetilde{\gamma}\left(\widetilde{z},\widetilde{\zeta}\right)^{\nicefrac{1}{q_{i}}},
\]
and thus, since $\widetilde{\varphi}_{2+i}(0)=0$,
\[
\left|\widetilde{\varphi}_{2+i}(t)\right|\lesssim\widetilde{\gamma}\left(\widetilde{z},\widetilde{\zeta}\right)^{\nicefrac{1}{q_{i}}}.
\]
Now let us denote $\varphi(t)=\left(\widetilde{\varphi}_{1}(t),\widetilde{\varphi}_{2}(t)\right)$
the projection of $\widetilde{\varphi}$ onto $\mathbb{C}^{2}$, and
let us write $\varphi'(t)=\sum_{i=1}^{4}c_{i}(t)\mathcal{Y}_{i}(\varphi(t))$.
We have to estimate the contribution of the coefficients $b_{j}$
to the coefficient $c_{i}$.

Suppose $j>4$. The contribution to $c_{3}$ and $c_{4}$ is null
and to $c_{1}$ and $c_{2}$ is bounded by $\frac{\partial h}{\partial w_{j'}}\widetilde{F}_{j'+2}\left(\widetilde{z},\widetilde{\gamma}\left(\widetilde{z},\widetilde{\zeta}\right)\right)^{-\nicefrac{1}{2}}$
(with an evident correspondence $j'\longleftrightarrow j$). As $\left|\frac{\partial h_{j'}}{\partial w_{j'}}\right|\lesssim\left|w_{j'}\right|^{2q_{j'}-1}\lesssim\widetilde{\gamma}\left(\widetilde{z},\widetilde{\zeta}\right)^{\frac{2q_{j'}-1}{2q_{j'}}}$,
this contribution is bounded by
\[
\widetilde{\gamma}\left(\widetilde{z},\widetilde{\zeta}\right)^{\frac{2q_{j'}-1}{2q_{j'}}}\widetilde{\gamma}\left(\widetilde{z},\widetilde{\zeta}\right)^{\frac{1}{2q_{j'}}}=\widetilde{\gamma}\left(\widetilde{z},\widetilde{\zeta}\right).
\]

If $j=4\mbox{ or }3$, the contribution to $c_{4}$ and $c_{3}$ is
bounded by $\widetilde{\gamma}\left(\widetilde{z},\widetilde{\zeta}\right)$,
and the contribution to $c_{1}$ and $c_{2}$ is bounded by 
\[
\sum\left|\frac{\partial h_{i}}{\partial w_{i}}\right|\widetilde{F}_{1}\left(\widetilde{z},\widetilde{\gamma}\left(\widetilde{z},\widetilde{\zeta}\right)\right)^{-\nicefrac{1}{2}}\lesssim\widetilde{\gamma}\left(\widetilde{z},\widetilde{\zeta}\right).
\]

When $j=1\mbox{ or }2$, the contribution is bounded by $\widetilde{\gamma}\left(\widetilde{z},\widetilde{\zeta}\right)$.

This proves the lemma.
\end{proof}

\subsubsection{\label{sec:C2-Pointwise-estimate-of-Bergman-kernel}Pointwise estimate
of the Bergman kernel}
\begin{stthm}
\label{thm:C-2-Estimates-Weighted-Bergman-Kernel}Assume that $\Omega$
is pseudo-convex of finite type in $\mathbb{C}^{2}$ and that the
hypotheses on $\rho$, $h$, stated at the beginning of the section
are satisfied. Let $L$ be the complex tangent vector field to $\rho$
defined by $L=\frac{\partial\rho}{\partial z_{2}}\frac{\partial}{\partial z_{1}}-\frac{\partial\rho}{\partial z_{1}}\frac{\partial}{\partial z_{2}}$
and $N$ be the normal one such that $N\rho\equiv1$ in a small neighborhood
$U$ of $\partial\Omega$. Let $\mathcal{L}$ be a list of vector
fields belonging to $\left\{ L,\overline{L},N,\overline{N}\right\} $.

Let us denote by $K_{\omega}^{\Omega}$ the Bergman kernel of $L_{\omega}^{2}(\Omega)$
for the weight $\omega$. Then for sufficiently close points $p_{1}$
and $p_{2}$ in $U$, we have the following estimate:
\begin{eqnarray*}
\left|\mathcal{L}K_{\omega}^{\Omega}\left(p_{1},p_{2}\right)\right| & \leq & C_{\left|\mathcal{L}\right|}\left(\frac{1}{\delta\left(p_{1},p_{2}\right)^{2}}\right)^{1+\nicefrac{l_{N}}{2}}F_{L}^{1+\nicefrac{l_{L}}{2}}\left(p_{1},\delta\left(p_{1},p_{2}\right)\right)\prod_{j=1}^{m}\left(\frac{1}{\delta\left(p_{1},p_{2}\right)}\right)^{\nicefrac{1}{q_{j}}}\\
 & \simeq & C_{\left|\mathcal{L}\right|}\frac{\left(\frac{1}{\delta\left(p_{1},p_{2}\right)^{2}}\right)^{\nicefrac{l_{N}}{2}}F_{L}^{\nicefrac{l_{L}}{2}}\left(p_{1},\delta\left(p_{1},p_{2}\right)\right)}{\mbox{Vol}_{\omega}\left(B\left(p_{1},\delta\left(p_{1},p_{2}\right)\right)\right)},
\end{eqnarray*}
where $F_{L}$ is the weight associated to $L$, $l_{L}$ (resp. $l_{N}$)
denotes the number of times $L$ or $\overline{L}$ (resp. $N$ or
$\overline{N}$) appears in the list $\mathcal{L}$, $\delta\left(p_{1},p_{2}\right)=\left|\rho\left(p_{1}\right)\right|+\left|\rho\left(p_{2}\right)\right|+\delta_{\Omega}\left(p_{1},p_{2}\right)$,
$\delta_{\Omega}\left(p_{1},p_{2}\right)=\gamma\left(\pi\left(p_{1}\right),\pi\left(p_{2}\right)\right)$,
$\gamma$ being the pseudo-distance on $\partial\Omega$, and $B\left(p_{1},\delta\left(p_{1},p_{2}\right)\right)$
the associated pseudo-ball, of the geometry describe in \cite{Charpentier-Dupain-Geometery-Finite-Type-Loc-Diag}
and $\mbox{Vol}_{\omega}$ denotes the volume with respect to the
measure $\omega(z)d\lambda(z)$.\end{stthm}
\begin{proof}
We use (3) of \lemref{Relation-Bergman-omega-omega-tilde} and the
sharp estimates on $K_{B}^{\widetilde{\Omega}}$ deduced from \cite{Charpentier-Dupain-Geometery-Finite-Type-Loc-Diag}.

Suppose $\left|\mathcal{L}\right|=0$. We have $\left|K_{B}^{\widetilde{\Omega}}\left(\widetilde{z},\widetilde{\zeta}\right)\right|\lesssim\mbox{Vol}\left(\widetilde{B}\left(\widetilde{z},\widetilde{\delta}\left(\widetilde{z},\widetilde{\zeta}\right)\right)\right)^{-1}$,
by \cite[Main Theorem on the Bergman kernel, part II, p. 77]{Charpentier-Dupain-Geometery-Finite-Type-Loc-Diag},
with, by \lemref{Lemma-2_kernel_n=00003D2}, 
\[
\widetilde{\delta}\left(\widetilde{z},\widetilde{\zeta}\right)=\left|r\left(\widetilde{z}\right)\right|+\left|r\left(\widetilde{\zeta}\right)\right|+\delta_{\widetilde{\Omega}}\left(\widetilde{z},\widetilde{\zeta}\right)\simeq\left|\rho(z)\right|+\left|\rho(\zeta)\right|+\delta_{\Omega}(z,\zeta)=:\delta,
\]
and, by \cite[Section 3]{Charpentier-Dupain-Geometery-Finite-Type-Loc-Diag},$\left|K_{B}^{\widetilde{\Omega}}\left(\widetilde{z},\widetilde{\zeta}\right)\right|\lesssim\mbox{Vol}\left(\widetilde{B}\left(\widetilde{\pi(z)},\widetilde{\delta}\left(\widetilde{z},\widetilde{\zeta}\right)\right)\right)^{-1}$.

The estimates of the functions $F_{i}$ and $\widetilde{F}_{i}$ in
the proofs of the lemmas show that
\[
\mbox{Vol}\left(\widetilde{B}\left(\widetilde{\pi(z)},\delta\right)\right)\simeq\delta^{2}\widetilde{F}_{1}\left(\widetilde{\pi(z)},\delta\right)\prod_{k=1}^{m}\delta^{\nicefrac{1}{q_{k}}}\simeq\delta^{2+\sum_{k}\nicefrac{1}{q_{k}}}F_{1}(z,\delta).
\]

To finish the proof, we have to estimate $\mbox{Vol}\left(B(z,\delta)\right)=\int_{B(z,\delta)}\omega(\xi)dV(\xi)$:
\begin{eqnarray*}
\omega(\xi) & = & \mbox{Vol}\left\{ w\mbox{ such that }h(w)<-\rho(\xi)\right\} \\
 & \simeq & \prod\mbox{Vol}\left\{ w_{i}\mbox{ such that }h_{i}\left(w_{i}\right)<\rho(\xi)\right\} \simeq\prod(-\rho(\xi))^{\sum\nicefrac{1}{q_{i}}}.
\end{eqnarray*}
Then, using the fact that $\xi\in B(z,\delta)$ and $\left|\rho(z)\right|\lesssim\delta$
imply $\left|\rho(\xi)\right|\lesssim\delta$ (see \cite{Charpentier-Dupain-Geometery-Finite-Type-Loc-Diag,CD08}),
we obtain
\[
\mbox{Vol}\left(B(z,\delta)\right)\simeq\delta^{2+\sum\nicefrac{1}{q_{i}}}F_{1}(z,\delta)^{-1},
\]
which finishes the proof when $\left\Vert \mathcal{L}\right\Vert =0$.

When $\left|\mathcal{L}\right|\geq1$, the proof is done similarly
using Lemmas \ref{lem:Lemma-1_kernel_n=00003D2} and \ref{lem:Lemma-2_kernel_n=00003D2}
and the inequalities on the derivatives of $K_{B}^{\widetilde{\Omega}}$
given in \cite{Charpentier-Dupain-Geometery-Finite-Type-Loc-Diag}.\end{proof}
\begin{rem*}
This proof easily generalizes in higher dimensions $n$ when the Levi
form of $\rho$ has a rank $\geq n-2$.
\end{rem*}

\subsection{\label{sec:convex-case}The case of convex domains of finite type
in $\mathbb{C}^{n}$}

Now we assume that the function $h$ satisfies \hyporef{Hypothesis-V}
of \secref{Hypothesis_on_function_h} (for example $h(w)=\sum\left|w_{i}\right|^{2q_{i}}$,
$w_{i}\in\mathbb{C}$).

\subsubsection{Choice of the defining function and geometry of \texorpdfstring{$\widetilde{\Omega}$}{Omega-tilde}}

Because of \remref{geom-sep-convex-domains}, we have to choose a
special defining function to obtain useful properties on $\widetilde{\Omega}$.

Let $g$ be the gauge function for $\Omega$. Then $\rho=g^{4}e^{1-\nicefrac{1}{g}}-1$
is a smooth convex defining function for $\Omega$ which is of finite
type in a neighborhood of $\partial\Omega$. Thus the domain $\widetilde{\Omega}=\left\{ \left(z,w\right)\in\mathbb{C}^{n}\times\mathbb{C}^{m}\mbox{ such that }\rho(z)+h(w)<0\right\} $
is smooth, convex and of finite type in a neighborhood of $\partial\Omega\times\left\{ 0\right\} $.

To get a useful estimate of the Bergman kernel of $\widetilde{\Omega}$,
we need a precise comparison between the geometries of $\partial\Omega$
and of $\partial\widetilde{\Omega}$ near the points of $\partial\Omega\times\left\{ 0\right\} $.

Let $P_{0}\in\partial\Omega$, $\widetilde{P_{0}}=\left(P_{0},0\right)\in\partial\widetilde{\Omega}$
and $\delta>0$ sufficiently small. We now investigate extremal bases
at $P_{0}$ and $\widetilde{P_{0}}$ (in the sense of \cite{McNeal-convexes-94,Hef04}).
\begin{sllem}
\label{lem:extremal-coord-syst-Omega-Omege-tilde_Convex}Let $\left(z_{1}^{\delta},\ldots,z_{n}^{\delta}\right)$
be a $\delta$-extremal coordinate system at $P_{0}$. Then it is
possible to choose a $\delta$-extremal coordinate system at $\widetilde{P_{0}}$,
$\left(\widetilde{z}_{1}^{\delta},\ldots,\widetilde{z}_{m+n}^{\delta}\right)$,
such that, for $1\leq i\leq n$, $\widetilde{z_{i}}^{\delta}=\left(z_{i}^{\delta},0\right)$.\end{sllem}
\begin{proof}
Let
\[
\widetilde{H_{\delta}}=\left\{ \widetilde{P}\in\mathbb{C}^{n+m}\mbox{ such that }r\left(\widetilde{P}\right)=\delta\right\} .
\]
Let $\widetilde{P_{1}}\in\widetilde{H_{\delta}}$ such that $\left|\widetilde{P_{1}}-\widetilde{P_{0}}\right|$
is the euclidean distance $d\left(\widetilde{P_{0}},\widetilde{H_{\delta}}\right)$
from $\widetilde{P_{0}}$ to $\widetilde{H_{\delta}}$. Let $Q_{1}$
be the projection of $\widetilde{P_{1}}$ to $\mathbb{C}^{n}$ so
that $\widetilde{P_{1}}=\left(Q_{1},w\right)$. We have 
\[
\delta-\rho\left(Q_{1}\right)=h(w)\asymp\sum\left|w_{i}\right|^{2q_{i}},
\]
and, by the condition on the $q_{i}$ (recall that $h$ satisfies
\hyporef{Hypothesis-V} and then $q_{i}>\mathrm{typ}(\Omega)=\tau$)
we obtain
\[
\left|w\right|^{2}\gtrsim\left(\delta-\rho\left(Q_{1}\right)\right)^{\nicefrac{2}{2\tau+2}}.
\]

On the other hand, the geometric properties of $\Omega$ show that
there exists
\[
Q_{2}\in H_{\delta}=\left\{ P\in\mathbb{C}^{n}\mbox{ such that }\rho(P)=\delta\right\} 
\]
such that the distance $d\left(Q_{2},H_{\delta}\right)$ from $Q_{2}$
to $H_{\delta}$ is less than $C\left(\delta-\rho\left(Q_{1}\right)\right)^{\nicefrac{1}{\tau}}$
(with a constant $C$ independent of $\delta$), and, by the definition of $\widetilde{P_{1}}$,
\[
\left|P_{1}-Q_{1}\right|^{2}+c\left(\delta-\rho\left(Q_{1}\right)\right)^{\nicefrac{2}{2\tau+2}}\leq\left(\left|P_{1}-Q_{1}\right|+C\left(\delta-\rho\left(Q_{1}\right)\right)^{\nicefrac{1}{\tau}}\right)^{2}
\]
which implies, for $\delta$ small enough (depending on $c$, $C$
and $\tau$, i.e. on $\Omega$), $\rho\left(Q_{1}\right)=\delta$
and we can choose $\widetilde{z}_{1}^{\delta}=\left(z_{1}^{\delta},0\right)$.

Define now $\widetilde{H}_{2,\delta}$ as the intersection of $\widetilde{H}_{\delta}$
with the affine complex space orthogonal to $\widetilde{z}_{1}^{\delta}$
passing through $\widetilde{P}_{0}$. Let $\widetilde{P}_{2}$ be
such that $\left|\widetilde{P}_{2}-\widetilde{P}_{0}\right|$ is the
euclidean distance from $\widetilde{P}_{0}$ to $\widetilde{H}_{2,\delta}$.
Let $w_{0}\in\partial\Omega$ and $U$ be a small neighborhood of
$w_{0}$. Arguing as before, it is easy to show that $\widetilde{P}_{2}\in\partial\Omega\times\left\{ 0\right\} $
and we can choose $\widetilde{z}_{2}^{\delta}=\left(z_{2}^{\delta},0\right)$.
The proof is finished by induction.\end{proof}
\begin{cor*}
Let $L_{1},\ldots,L_{n}$ be the $\delta$-extremal basis of vector
fields associated to the $\delta$-extremal coordinate system at the
point $P_{0}\in\partial\Omega\cap U$ defined in \lemref{extremal-coord-syst-Omega-Omege-tilde_Convex}
(see \cite{CD08}). Then the basis $\widetilde{L}_{1},\ldots,\widetilde{L}_{n+m}$
defined by:
\begin{enumerate}
\item for $1\leq i\leq n$, $L_{i}=\widetilde{L}_{i}$,
\item for $1\leq j\leq m$, $\widetilde{L}_{n+j}=\frac{\partial}{\partial w_{j}}-\beta_{n+j}\frac{\partial}{\partial Z_{1}}$,
where $Z_{1}$ is the complex normal to $\partial\Omega$ at the point
$w_{0}$, $\beta_{n+j}$ being so that $\widetilde{L}_{n+j}$ is tangent
to $\partial\widetilde{\Omega}$,
\end{enumerate}
is $\delta$-extremal at $\widetilde{P}_{0}$.\end{cor*}
\begin{proof}
For the point (1), note that, for $i\geq2$, $L_{i}=\frac{\partial}{\partial z_{i}}-\beta_{i}\frac{\partial}{\partial Z_{1}}$
(see \cite[Section 7.1]{CD08}) (recall that $L_{1}=\widetilde{L}_{1}=N$).
For (2), without loss of generality, we can assume $q_{j+1}\geq q_{j}$,
$1\leq j\leq m-1$, and the result is trivial if $h_{i}\left(w_{i}\right)=\left|w_{i}\right|^{2q_{i}}$
and ``easy'' to prove in the general case using \cite{CD08}.
\end{proof}

Let $F_{i}$ and $\widetilde{F}_{i}$ be the weights defined with
the vector fields $L_{i}$ and $\widetilde{L}_{i}$. Then
\begin{sllem}
For $1\leq i\leq n$, $F_{i}\left(z,\delta\right)=\widetilde{F}_{i}\left(\widetilde{z},\delta\right)$
and, for $1\leq j\leq m$, $\widetilde{F}_{n+j}\left(\widetilde{z},\delta\right)\simeq\left(\frac{1}{\delta}\right)^{\nicefrac{1}{q_{j}}}$,
for $z\in U$ and $\widetilde{z}=\left(z,0\right)$.\end{sllem}
\begin{proof}
The first part is a trivial consequence of the preceding corollary
and the second is proved, as in the case of dimension $2$, noting
that $\beta_{n+j}=-\nicefrac{\frac{\partial h_{j}\left(w_{j}\right)}{\partial w_{j}}}{\frac{\partial\rho}{\partial Z_{1}}}$,
with $\frac{\partial\rho}{\partial Z_{1}}$ $\mathcal{C}^{\infty}$
and close to $1$ for $\delta$ small.
\end{proof}

\subsubsection{\label{sec:Convex-Pointwise-estimate-of-Berg-Ker}Pointwise estimate
of the Bergman kernel}
\begin{stthm}
\label{thm:Convex-Estimates-Weighted-Bergman-Kernel}Assume $\Omega$
is convex of finite type in $\mathbb{C}^{n}$ and that the hypothesis
on $\rho$, $h$, and $\omega$ stated at the beginning of the section
are satisfied. Let $w_{0}$ be a boundary point of $\Omega$ and $U$
a small neighborhood of $w_{0}$. Let $N$ be the complex normal to
$\partial\Omega$ (i.e. $N\rho\equiv1$ in a neighborhood $U$ of
$\partial\Omega$). Let $p_{1}$ and $p_{2}$ be two points in $U$
and $\delta_{\Omega}\left(p_{1},p_{2}\right)$ as in \thmref{C-2-Estimates-Weighted-Bergman-Kernel}.
Let $\left\{ L_{2},\ldots,L_{n}\right\} $ be a $\delta\left(p_{1},p_{2}\right)$-extremal
basis associated to $\rho$ at the point $p_{1}$ (with $\delta\left(p_{1},p_{2}\right)=\left|\rho\left(p_{1}\right)\right|+\left|\rho\left(p_{2}\right)\right|+\delta_{\Omega}\left(p_{1},p_{2}\right)$).
Let us denote $L_{1}=N$. Let $\mathcal{L}$ be a list of vector fields
belonging to $\left\{ L_{1},\overline{L_{1}},\ldots,L_{n},\overline{L_{n}},N,\overline{N}\right\} $.
Let $K_{\omega}^{\Omega}$ be the Bergman kernel of $L_{\omega}^{2}(\Omega)$
for the weight $\omega$. Then
\begin{eqnarray*}
\left|\mathcal{L}K_{\omega}^{\Omega}\left(p_{1},p_{2}\right)\right| & \leq & C_{\left|\mathcal{L}\right|}\left(\frac{1}{\delta\left(p_{1},p_{2}\right)^{2}}\right)^{1+\nicefrac{l_{N}}{2}}\mathcal{F}^{1+\nicefrac{\mathcal{L}}{2}}\left(p_{1},\delta\left(p_{1},p_{2}\right)\right)\prod_{j=1}^{m}\left(\frac{1}{\delta\left(p_{1},p_{2}\right)}\right)^{\nicefrac{1}{q_{j}}}\\
 & \simeq & C_{\left|\mathcal{L}\right|}\frac{\left(\frac{1}{\delta\left(p_{1},p_{2}\right)^{2}}\right)^{\nicefrac{l_{N}}{2}}\mathcal{F}^{\nicefrac{\mathcal{L}}{2}}\left(p_{1},\delta\left(p_{1},p_{2}\right)\right)}{\mbox{Vol}_{\omega}\left(B\left(p_{1},\delta\left(p_{1},p_{2}\right)\right)\right)},
\end{eqnarray*}
where $l_{N}$ denotes the number of times $N$ or $\overline{N}$
appears in the list $\mathcal{L}$, $\mathcal{F}^{1+\nicefrac{\mathcal{L}}{2}}\left(p_{1},\delta\left(p_{1},p_{2}\right)\right)=\prod_{i=2}^{n}\mathcal{F}_{L_{i}}^{1+\nicefrac{l_{i}}{2}}$,
$l_{i}$ being the number of times $L_{i}$ or $\overline{L_{i}}$
appears in the list $\mathcal{L}$ and $\mbox{Vol}_{\omega}$ denotes
the volume with respect to the measure $\omega(z)d\lambda(z)$.\end{stthm}
\begin{proof}
The construction made before shows that the estimate is immediate
because, the exponential map being a local diffeomorphism (\cite[p. 75]{Charpentier-Dupain-Geometery-Finite-Type-Loc-Diag}),
the fact that $L_{i}=\widetilde{L}_{i}$, $1\leq i\leq n$ (corollary
of \lemref{extremal-coord-syst-Omega-Omege-tilde_Convex}), implies
$\delta_{\Omega}\left(p_{1},p_{2}\right)=\delta_{\widetilde{\Omega}}\left(p_{1},p_{2}\right)$.
\end{proof}

\subsection{\label{sec:Proof-of-Main-Theorem}Proof of \texorpdfstring{\thmref{Sobolev-Lp_Lip_Bergman}}{Theorem 1.1}}

In the two cases we consider here, $\partial\widetilde{\Omega}$ is
of finite type at every point of the form $\left(z,0\right)$. Then,
by Catlin's theorem (\cite{Cat87}), the results of \cite{Kohn-Nirenberg-1965}
show that the Neumann operator of $\widetilde{\Omega}$ is pseudolocal
at these points, and, the method introduced by N. Kerzman in \cite{Kerzman-Bergman}
proves that the restriction of the Bergman kernel of $\widetilde{\Omega}$
to $\left(\overline{\Omega}\times\left\{ 0\right\} \right)^{2}$ is
$\mathcal{C}^{\infty}$ outside the diagonal of $\left(\partial\Omega\times\left\{ 0\right\} \right)^{2}$.
Thus, the identity $K_{\omega}^{\Omega}\left(p_{1},p_{2}\right)=K^{\widetilde{\Omega}}\left(\left(p_{1},0\right),\left(p_{2},0\right)\right)$
implies that the estimates of Theorems \ref{thm:C-2-Estimates-Weighted-Bergman-Kernel}
and \ref{thm:Convex-Estimates-Weighted-Bergman-Kernel} are valid
everywhere.

These estimates, the hypothesis on $h$ (i.e. $h(w)\asymp\sum\left|w_{i}\right|^{2q_{i}}$),
an immediate generalization of Proposition 2.1 of \cite{BCG96} and
a standard application of H\"older inequality imply that $P_{\omega}^{\Omega}$
maps $L^{p}\left(\Omega,(-\rho)^{\alpha}d\lambda\right)$
continuously into itself for $-1<\alpha<p\left(1+\sum\frac{1}{q_{i}}\right)-1$.

The Lipschitz estimate is also standard.

Now, choosing the special function $h(w)=\sum\left|w_{i}\right|^{2q_{i}}$,
$w_{i}\in\mathbb{C}$, the weight $\omega$ is equal to $C(-\rho)^{\sum\nicefrac{1}{q_{i}}}$,
and \thmref{Sobolev-Lp_Lip_Bergman} follows.
\begin{rem*}
Note that same method gives trivially \thmref{Sobolev-Lp_Lip_Bergman}
for pseudo-convex decoupled domains of finite type in $\mathbb{C}^{n}$.
\end{rem*}

\bibliographystyle{amsalpha}

\providecommand{\bysame}{\leavevmode\hbox to3em{\hrulefill}\thinspace}
\providecommand{\MR}{\relax\ifhmode\unskip\space\fi MR }
\providecommand{\MRhref}[2]{%
  \href{http://www.ams.org/mathscinet-getitem?mr=#1}{#2}
}
\providecommand{\href}[2]{#2}

\end{document}